\documentclass[aop,MSNbibl,dvips]{arximspdf}
\usepackage{mathbh}

\doi{10.1214/12-AOP763} 
\volume{41}
\issue{5}
\pubyear{2013}
\firstpage{3306}
\lastpage{3344}

\makeatletter

\newproclaim{hypothesis}{Hypothesis}
\newcommand{\rrVert}{\Vert}
\newcommand{\rrvert}{\vert}
\newcommand{\llVert}{\Vert}
\newcommand{\llvert}{\vert}

\newtheorem{theorem}{Theorem}
\newtheorem{lemma}[theorem]{Lemma}
\newtheorem{proposition}[theorem]{Proposition}

\newproclaim{remark}[theorem]{Remark}

\newcommand{\one}{\mathbh{1}}
\renewcommand{\P}{\mathbb P}
\newcommand{\E}{\mathbb E}
\newcommand{\N}{\mathbb N}
\newcommand{\Q}{\mathbb Q}
\newcommand{\R}{\mathbb R}

\makeatother

\begin{document}
\begin{frontmatter}

\title{Strong uniqueness for stochastic evolution equations
in Hilbert spaces perturbed by a bounded measurable drift}
\runtitle{Uniqueness for stochastic evolution equations}

\begin{aug}
\author[A]{\fnms{G.} \snm{Da Prato}\ead[label=e1]{g.daprato@sns.it}},
\author[B]{\fnms{F.} \snm{Flandoli}\ead[label=e2]{flandoli@dma.unipi.it}},
\author[C]{\fnms{E.} \snm{Priola}\corref{}\thanksref{t1}\ead[label=e3]{enrico.priola@unito.it}}
\and
\author[D]{\fnms{M.} \snm{R\"{o}ckner}\thanksref{t2}\ead[label=e4]{roeckner@mathematik.uni-bielefeld.de}}
\runauthor{Da Prato, Flandoli, Priola and R\"{o}ckner}
\affiliation{Scuola Normale Superiore, Universit\`a di Pisa,
Universit\`a di Torino
and University of Bielefeld}
\address[A]{G. Da Prato\\
Scuola Normale Superiore\\
Piazza dei Cavalieri 7\\
56126 Pisa\\
Italy\\
\printead{e1}} 
\address[B]{F. Flandoli\\
Dipartimento di Matematica Applicata\hspace*{12.3pt}\\
\quad``U. Dini''\\
Universit\`a di Pisa\\
V.le B. Pisano 26/b\\
56126 Pisa\\
Italy\\
\printead{e2}}
\address[C]{E. Priola\\
Dipartimento di Matematica\\
Universit\`a di Torino\\
via Carlo Alberto 10, Torino\\
Italy\\
\printead{e3}}
\address[D]{M. R\"{o}ckner\\
Faculty of Mathematics\\
Bielefeld University\\
D-33501 Bielefeld\\
Germany\\
\printead{e4}}
\end{aug}

\thankstext{t1}{Supported in part by the
M.I.U.R. research project Prin 2008 ``Deterministic and stochastic
methods in the study of evolution problems.''}

\thankstext{t2}{Supported by the DFG through IRTG 1132 and CRC 701
and the I. Newton Institute, Cambridge, UK.}

\received{\smonth{9} \syear{2011}}
\revised{\smonth{1} \syear{2012}}

%
\begin{abstract}
We prove pathwise (hence strong) uniqueness of solutions
to stochastic evolution equations in Hilbert spaces with merely
measurable bounded drift and cylindrical Wiener noise, thus
generalizing Veretennikov's fundamental result on $\R^d$ to infinite
dimensions. Because Sobolev regularity results implying continuity
or smoothness of functions do not hold on infinite-dimensional
spaces, we employ methods and results developed in the study of
Malliavin--Sobolev spaces in infinite dimensions. The price we pay is
that we can prove uniqueness for a large class, but not for every
initial distribution. Such restriction, however, is common in
infinite dimensions.
\end{abstract}

%

\end{frontmatter}

\section{Introduction}\label{sec1}

We consider the following abstract stochastic differential equation
in a
separable Hilbert space $H$:
%
%
\begin{equation}
\label{SPDE} dX_{t}=\bigl(AX_{t}+B(X_{t})
\bigr)\,dt+dW_{t},\qquad X_{0}=x\in H,
\end{equation}
where $A\dvtx D(A)\subset H\rightarrow H$ is self-adjoint, negative
definite and such that $(-A)^{-1 + \delta}$, for some $\delta\in
(0,1)$,
is
of trace class,
$B\dvtx H\rightarrow H$ and $W=(W_{t})$ is a cylindrical Wiener process.
About $B$, we only assume that it is \textit{Borel measurable and
bounded}.%
\[
B\in{B}_{b}(H,H).
\]
Our aim is to prove \textit{pathwise uniqueness} for (\ref{SPDE}),
thus gaining an infinite-dimensional generalization\vadjust{\goodbreak} of the famous
fundamental
result of Veretennikov~\cite{Ver} in the case $H = \R^d$.
We refer to~\cite{Zv74} and~\cite{tre} for the case $H= \R$
as well as to the generalizations of~\cite{Ver} to unbounded drifts
in~\cite{KR05,Za} and also to the references therein; see~\cite{Gyongy-Krylov,GyongyMartinez}.
We note that~\cite{tre} also includes the case of $\alpha$-stable
noise, $\alpha\ge1$, which in turn was extended to $\R^d$ in
\cite{P}.

Explicit cases of parabolic stochastic partial differential equations,
with
space--time white noise in space-dimension one, have been solved on
various levels of generality for the drift by Gy\"{o}ngy and
coworkers, in a series of papers; see
\cite{AlabertGyongy,Gyongy,GyongyNualart,GyongyPardoux}
and the references therein. The difference of the
present paper with
respect to these works is that we obtain a general abstract result,
applicable, for instance, to systems of parabolic equations or
equations with differential operators of higher order than two. As
we shall see, the price to pay for this generality is a restriction
on the initial conditions. Indeed, using that for $B=0$ there
exists a unique nondegenerate (Gaussian) invariant measure $\mu$,
we will prove strong uniqueness for $\mu$-a.e. initial $x \in H$ or
random $H$-valued $x$ with distribution absolutely continuous with
respect to $\mu$.

At the abstract level, this work generalizes
\cite{DF} devoted to the case where $B$ is bounded and in addition
H\"{o}lder continuous, but with no restriction on the initial
conditions. To prove our result we use some ideas from
\cite{FF,FGP,Fla,DF} and~\cite{KR05}.

The extension of Veretennikov's result~\cite{Ver} and also of
\cite{KR05} to infinite dimensions has resisted various attempts of its
realization for many years. The reason is that the finite-dimensional
results heavily depend on advanced parabolic Sobolev regularity results
for solutions to the corresponding Kolmogorov equations. Such
regularity results, leading to continuity or smoothness of the
solutions, however, do not hold in infinite dimensions. A technique
different from~\cite{Ver} is used in~\cite{FGP}; see also~\cite{DF,FF}
and~\cite{P}. This technique allows us to prove uniqueness for
stochastic equations with time independent coefficients by merely using
elliptic (not parabolic) regularity results. In the present paper we
succeed in extending this approach to infinite dimensions, exploiting
advanced regularity results for elliptic equations in
Malliavin--Sobolev spaces with respect to a Gaussian measure on Hilbert
space. To the best of our knowledge this is the first time that an
analogue of Veretennikov's result has been obtained.

Given a filtered probability space $ (
\Omega, {\mathcal F},
({\mathcal F}_{t}), \P) $, a
cylindrical Wie\-ner process $W$ and an ${\mathcal F}_0$-measurable
r.v. $x$, we call \textit{mild solution} to the Cauchy problem
(\ref{SPDE}) a continuous ${\mathcal F}_{t}$-adapted $H$-valued
process $X = (X_t)$ such that
%
%
\begin{equation}
\label{mild} X_{t}=e^{tA}x+\int_{0}^{t}e^{ ( t-s ) A}B
( X_{s} ) \,ds+\int_{0}^{t}e^{ ( t-s ) A}\,dW_{s}.%
\end{equation}
Existence of mild solutions on some filtered probability space is
well known; see Chapter 10 in~\cite{DZ} and also Appendix~\ref{aapA.1}. Our
main result is:
%
%
\begin{theorem}
\label{maintheorem} Assume Hypothesis~\ref{hypo1}. For $\mu$-a.e.
(deterministic) $x\in H$, there is a unique (in the pathwise sense)
mild solution of the Cauchy problem~(\ref{SPDE}).

Moreover, for every ${\mathcal F}_{0}$-measurable $H$-valued r.v.
$x$ with law $\mu_{0}$ such that $\mu_{0}\ll\mu$ and
\[
\int_{H} \biggl( \frac{d\mu_{0}}{d \mu} \biggr)^{\zeta}\,d
\mu<\infty
\]
for some
$\zeta>1$, there is also a unique mild solution of the Cauchy
problem.
\end{theorem}

The proof, performed in Section~\ref{sectionproof}, uses regularity
results for elliptic equations in Hilbert spaces, given in Section~\ref{sec2}
where we also establish an It\^o type formula involving $u(X_t)$
with $u$ in some Sobolev space associated to~$\mu$. In comparison
with the finite-dimensional case (cf.~\cite{KR05}),
to prove such an It\^o formula,
we do not only need
analytic regularity results,
but also the fact that all transition probability functions
associated
with (\ref{mild}) are absolutely continuous with respect to $\mu$.
This result heavily depends on an infinite-dimensional version of
Girsanov's theorem. Though, also under our conditions, this is a
``folklore result'' in the field; it seems hard to find an
accessible reference in the literature. Therefore, we include a
complete proof of the version we need in the \hyperref[app]{Appendix}
for the
convenience of the reader.

Concerning the proof of Theorem~\ref{maintheorem} given in
Section~\ref{sec3}, we remark that, in comparison to the finite-dimensional
case (see, in particular,~\cite{Fedrizzi} and~\cite{FF}), it is
necesary to control infinite series of second derivatives of
solutions to Kolmogorov equations which is much more elaborate.

Examples are given in Section~\ref{sectionexamples}.

\subsection{Assumptions and preliminaries}\label{sec1.1}

We are given a real separable Hilbert space $H$ and denote its
norm and inner product by $\llvert\cdot\rrvert$ and
$ \langle\cdot, \cdot\rangle$, respectively. We follow
\cite{DZ,DZ1,D} and assume:

\begin{hypothesis}\label{hypo1}
$A\dvtx D(A)\subset H\to H$ is a negative
definite
self-adjoint operator and
$(-A)^{-1 + \delta}$, for some $\delta\in(0,1)$,
is of trace class.
\end{hypothesis}

Since $A^{-1}$ is compact, there exists an orthonormal basis $(e_k)$
in $H$ and a sequence of positive numbers $(\lambda_k)$ such that
%
%
\begin{equation}
\label{e1a} Ae_k=-\lambda_k e_k,\qquad
k\in\mathbb N.
\end{equation}
Recall that $A$ generates an analytic semigroup $e^{tA}$ on $H$ such
that $e^{tA} e_k = e^{- \lambda_k t } e_k$. We will consider a
cylindrical Wiener process $W_t$ with respect to the previous
basis~$(e_k)$. The process $W_t$ is formally given by ``$W_t = \sum_{k
\ge1} \beta_k(t) e_k$'' where $\beta_k(t)$ are independent,
one-dimensional Wiener process; see~\cite{DZ} for more details.

By $R_t$ we denote the Ornstein--Uhlenbeck semigroup in $B_b(H)$ (the
Banach space of Borel and bounded real functions endowed with the
essential supremum norm $\| \cdot\|_0$) defined as
%
%
\begin{equation}
\label{e1} R_t\varphi(x)=\int_H\varphi(y)N
\bigl(e^{tA}x,Q_t\bigr) (dy), \qquad\varphi\in
B_b(H),
\end{equation}
where $N(e^{tA}x,Q_t)$ is the Gaussian measure in $H$ of mean
$e^{tA}x$ and covariance operator $Q_t$ given by
%
%
\begin{equation}
\label{e2} Q_t=-\tfrac12 A^{-1}\bigl(I-e^{2tA}
\bigr),\qquad t\ge0.
\end{equation}
We notice that $R_t$ has a unique invariant measure $\mu:=N(0,Q)$
where $Q=-\frac12 A^{-1}$. Moreover, since under the previous
assumptions the Ornstein--Uhlenbeck semigroup is strong Feller and
irreducible, we have by Doob's theorem that, for any $t>0$, $x \in
H$, the measures $N(e^{tA}x,Q_t)$ and\vspace*{1pt} $ \mu$ are equivalent; see
\cite{DZ1}. On the other hand, our assumption that $(-A)^{-1+
\delta}$ is trace class guarantees that the OU process,
%
%
\begin{equation}
\label{ou11} Z_t = Z(t,x)= e^{tA} x + \int
_0^t e^{(t-s)A} \,dW_s
\end{equation}
has a continuous $H$-valued version.

If $H$ and $K$ are separable Hilbert spaces, the Banach space
$L^p(H, \mu$, $ K)$, $p \ge1$, is defined to consist of equivalent
classes of measurable functions $f\dvtx H \to K$ such that $\int_H
|f|_K^p \mu(dx) < +\infty$ [if $K = \R$ we set $L^p(H, \mu, \R)=
L^p(H, \mu)$]. We also use the notation $L^p(\mu)$ instead of
$L^p(H, \mu, K)$ when no confusion may arise.

The semigroup $R_t$ can be uniquely extended to a strongly
continuous semigroup of contractions on $L^p(H,\mu)$, $p \ge1$,
which we still denote by $R_t$, whereas we denote by $L_p$ (or $L$
when no confusion may arise) its infinitesimal generator, which is
defined on smooth functions $\varphi$ as
\[
L \varphi(x) = \tfrac{1}{2} \mbox{Tr}\bigl(D^2 \varphi(x)\bigr)
+ \bigl\langle Ax, D\varphi(x)\bigr\rangle,
\]
where $D \varphi(x)$ and $D^2 \varphi(x)$ denote, respectively, the
first and second Fr\'echet derivatives of $\varphi$ at $x \in H$.
For Banach spaces $E$ and $F$ we denote by $C^{k}_b (E,F)$, $k \ge
1$, the Banach space of all functions $f\dvtx E \to F$ which are bounded
and Fr\'echet differentiable on $E$ up to the order $k \ge1$ with
all derivatives bounded and continuous. We also set $C^{k}_b (E, \R)
= C^{k}_b (E)$.

According to~\cite{DZ1}, for any $\varphi\in B_b(H)$ and any $t>0$
one has
$R_t\varphi\in C^{\infty}_b(H) = \bigcap_{k \ge1} C^k_b(H)$. Moreover,
%
%
\begin{equation}
\label{e3} \bigl\langle DR_t\varphi(x),h \bigr\rangle=\int
_H \bigl\langle\Lambda_t h,Q_t^{-1/2}
y\bigr\rangle\varphi\bigl(e^{tA}x+y\bigr)N(0,Q_t) (dy),\qquad
h\in H,\hspace*{-22pt}
\end{equation}
where $Q_t$ is defined in (\ref{e2}),
%
%
\begin{equation}
\label{e4} \Lambda_t=Q_t^{-1/2}e^{tA}=
\sqrt2 (-A)^{1/2}e^{tA}\bigl(I-e^{2tA}
\bigr)^{-1/2}
\end{equation}
and $y \mapsto\langle
\Lambda_t h,Q_t^{-1/2} y\rangle$ is a centered
Gaussian random variable
under $\mu_t = N(0, Q_t)$ with variance $ |\Lambda_t h|^2$
for any $t>0$; cf.~\cite{DZ}, Theorem 6.2.2. Since
\[
\Lambda_t e_k =\sqrt2 (\lambda_k)^{1/2}e^{-t\lambda_k}
\bigl(1-e^{-2t\lambda_k}\bigr)^{-1/2} e_k,
\]
we see that, for any $\epsilon\in[0,\infty)$, there exists
$C_\epsilon>0$ such that
%
%
\begin{equation}
\label{e5} \bigl\|(-A)^\epsilon\Lambda_t \bigr\|_{\mathcal L}\le
C_\epsilon t^{-1/2-\epsilon}.
\end{equation}
In the sequel $\| \cdot\|$ always denotes the \textit{Hilbert--Schmidt
norm}; on the other hand $\| \cdot\|_{\mathcal L}$ indicates the
\textit{operator norm}.


By (\ref{e3}) we deduce
%
%
\begin{equation}
\label{e6} \sup_{x \in H} \bigl|(-A)^\epsilon DR_t
\varphi(x)\bigr| = \bigl\|(-A)^\epsilon DR_t\varphi\bigr\|_0 \le
C_\epsilon t^{-1/2-\epsilon}\|\varphi\|_0,
\end{equation}
which by taking the Laplace transform yields, for
$\epsilon\in[0,1/2)$,
%
%
\begin{equation}
\label{e10b} 
\bigl\|(-A)^\epsilon D (\lambda-
L_2)^{-1} \varphi\bigr\|_{0}\le\frac{C_{1,\epsilon}}{\lambda^{
1/2-\epsilon}} \|
\varphi\|_{0}.
\end{equation}
Similarly, we find
%
%
\begin{equation}
\label{e10a} \bigl\|(-A)^\epsilon DR_t\varphi\bigr\|_{L^2(\mu)}\le
C_\epsilon t^{-1/2-\epsilon} \|\varphi\|_{L^2(\mu)}
\end{equation}
and
%
%
\begin{equation}
\label{e12b} \bigl\|(-A)^\epsilon D (\lambda- L_2)^{-1}
\varphi\bigr\|_{L^2(\mu)}\le\frac{C_{1,\epsilon}} {
\lambda^{1/2-\epsilon}} \|\varphi\|_{L^2(\mu)}.
\end{equation}
Recall that
the Sobolev space $W^{2,p}(H, \mu)$, $p \ge1$, is defined in
\cite{CG1}, Section 3,
as the completion of a suitable set of smooth functions
endowed with the Sobolev norm; see also~\cite{DZ},
Section 9.2, for the case $p=2$ and~\cite{S}.
Under our initial assumptions,
the following result can be found in
\cite{DZ1}, Section 10.2.1.

%
%
\begin{theorem}
\label{t1}
Let $\lambda>0$, $f\in L^2(H,\mu)$ and let $\varphi\in D(L_2)$ be
the solution of the equation
%
%
\begin{equation}
\label{e10} \lambda\varphi-L_2\varphi=f.
\end{equation}
Then $\varphi\in W^{2,2}(H,\mu)$, $(-A)^{1/2}D\varphi\in
L^{2}(H,\mu)$ and there exists a constant $C(\lambda)$
such that
%
%
\begin{eqnarray}
\label{e11}
&&
\|\varphi\|_{L^{2}(\mu)} + \biggl( \int_H
\bigl\|D^2 \varphi(x)\bigr\|^2 \mu(dx)
\biggr)^{1/2} + \bigl\|(-A)^{1/2}D\varphi\bigr\|_{L^{2}(\mu)}\nonumber\\[-8pt]\\[-8pt]
&&\qquad\le C\|f
\|_{L^2(\mu).}\nonumber
\end{eqnarray}
%
\end{theorem}
The following extension to $L^p(\mu)$, $p>1$ can be found in
Section 3 of~\cite{CG1}; see also~\cite{CG} and~\cite{MV}; a finite-dimensional
result analogous to this
for nonsymmetric OU operators was proved in
\cite{MPRS}.
%
%
\begin{theorem}
\label{t1cg}
Let $\lambda>0$, $f\in L^p(H,\mu)$ and let $\varphi\in D(L_p)$ be
the solution of the equation
%
%
\begin{equation}
\label{e10cg} \lambda\varphi-L_p \varphi=f.
\end{equation}
Then $\varphi\in W^{2,p}(H,\mu)$, $(-A)^{1/2}D\varphi\in
L^{p}(H,\mu;H)$ and there exists a constant $C= C(\lambda, p)$ such
that
%
%
\begin{eqnarray}
\label{e11cg}
&&
\|\varphi\|_{L^p(\mu)} + \biggl( \int_H
\bigl\|D^2 \varphi(x)\bigr\|^p \mu(dx)
\biggr)^{1/p} +\bigl\|(-A)^{1/2}D\varphi\bigr\|_{L^{p}(\mu)}\nonumber\\[-8pt]\\[-8pt]
&&\qquad\le C\|f
\|_{L^p(\mu)}.\nonumber
\end{eqnarray}
\end{theorem}

\section{Analytic results and an It\^o-type formula}\label{sec2}

\subsection{Existence and uniqueness for the Kolmogorov
equation}\label{sec2.1}

We are here concerned with the equation
%
%
\begin{equation}
\label{e12} \lambda u-L_2 u-\langle B,Du \rangle=f,
\end{equation}
where $\lambda>0$, $f\in B_b(H)$ and $B\in B_b(H,H)$.
%
%
\begin{remark} Since
the corresponding Dirichlet form
\[
{\mathcal E}(u,v):= \int_H \langle Du, Dv \rangle\,d\mu-
\int_H \langle B, Du \rangle v \,d\mu+ \lambda\int
_H uv \,d\mu,
\]
$u,v \in W^{1,2}(\mu)$, is weakly sectorial for $\lambda$ big
enough,
it follows by~\cite{MR}, Chapter I and Subsection 3e) in Chapter II,
that
(\ref{e12}) has a unique solution in $D(L_2)$. However, we need
more regularity for $u$.
\end{remark}
%
%
\begin{proposition}
\label{p2} Let $\lambda\ge\lambda_0$, where
%
%
\begin{equation}
\label{dt1} \lambda_0:=4\|B\|_0^2
C^2_{1,0}.
\end{equation}
Then there is a unique solution $u \in D(L_2)$ of (\ref{e12}) given
by
%
%
\begin{equation}
\label{e13} u= u_{\lambda}=(\lambda-L_2)^{-1}(I-T_\lambda)^{-1}f,
\end{equation}
where
%
%
\begin{equation}
\label{e14} T_\lambda\varphi:=\bigl\langle B,D(\lambda-L_2)^{-1}
\varphi\bigr\rangle.
\end{equation}
Moreover, $u \in C^1_b (H)$ with
%
%
\begin{equation}
\label{e15} \|u\|_0\le2\|f\|_0,\qquad
\bigl\|(-A)^\epsilon Du\bigr\|_0\le\frac{2C_{1,\epsilon}}{\lambda^{
1/2-\epsilon}} \|f
\|_0,\qquad \epsilon\in[0, 1/2),
\end{equation}
and, for any $p \ge2$, $u \in W^{2,p}(H, \mu)$ and,
for some $C=C(\lambda, p, \| B\|_0)$,
%
%
\begin{equation}
\label{cgp} \int_H \bigl\|D^2 u (x)
\bigr\|^p \mu(dx) \le C \int_H
\bigl|f(x)\bigr|^p \mu(dx).
\end{equation}
\end{proposition}
\begin{pf}
Setting $\psi:=\lambda u -L_2 u $, equation (\ref{e12}) reduces to
%
%
\begin{equation}
\label{e19b} \psi-T_\lambda\psi=f.
\end{equation}
If $\lambda\ge\lambda_0$ by (\ref{e12b}), we have
\[
\|T_\lambda\varphi\|_{L^2(\mu)}\le\tfrac12 \|\varphi
\|_{L^2(\mu)},\qquad \varphi\in L^2(\mu),
\]
so that (\ref{e19b}) has a unique solution given by
\[
\psi=(I-T_\lambda)^{-1}f.
\]
Consequently, (\ref{e12}) has a unique solution $u\in L^2(H,\mu)$
given by (\ref{e13}). The same argument in $B_b(H)$, using
(\ref{e10b}) instead of (\ref{e12b})
shows that
\[
\|T_\lambda\varphi\|_{0}\le\tfrac12 \|\varphi
\|_{0},\qquad \varphi\in B_b(H),
\]
and that $\psi\in B_b(H)$ and hence by (\ref{e13})
also $u \in B_b(H)$. In particular,
(\ref{e15}) is fulfilled by (\ref{e10b}). To prove the last
assertion we write $\lambda u-L_2 u = \langle B,Du \rangle+ f$ and
use estimate (\ref{e15}) with $\epsilon=0$ and Theorem~\ref{t1cg}.
\end{pf}

\subsection{Approximations}\label{sec2.2}
We are given two sequences $(f_n)\subset B_b(H)$ and $(B_n)\subset
B_b(H,H)$ such that:
%
%
\begin{eqnarray}
\label{e16} %
\mbox{(i)\quad\hspace*{1.5pt}} f_n(x)&\to& f(x),\qquad
B_n(x)\to B(x),\qquad \mu\mbox{-a.s.};
\nonumber\\[-8pt]\\[-8pt]
\mbox{(ii)\quad} \|f_n\|_0&\le& M,\qquad \|B_n
\|_0\le M.
\nonumber
\end{eqnarray}

%
\begin{proposition}
\label{p3} Let $\lambda\ge\lambda_0$, where
$\lambda_0$ is defined in (\ref{dt1}).
Then the equation
%
%
\begin{equation}
\label{e17} \lambda u_n-L u_n-\langle
B_n,Du_n \rangle=f_n
\end{equation}
has a unique solution $u_n\in C_b^1(H) \cap D(L_2)$ given by
%
%
\begin{equation}
\label{e18} u_n=(\lambda-L)^{-1}(I-T_{n,\lambda})^{-1}
f_n,
\end{equation}
where
%
%
\begin{equation}
\label{e19} T_{n,\lambda} \varphi:=\bigl\langle B_n, D(
\lambda-L_2)^{-1} \varphi\bigr\rangle.
\end{equation}
Moreover,
for any $\epsilon\in[0, 1/2)$,
with
constants independent of $n$,
%
%
\begin{equation}
\label{e20} \|u_n\|_0\le2M,\qquad \bigl\|(-A)^\epsilon
Du_n\bigr\|_0 \le\frac{2C_{1,\epsilon}}{\lambda^{1/2-\epsilon}}M.
\end{equation}
Finally,
we have $u_n\to u$,
and
$Du_n\to Du$, in
$L^2(\mu)$, where $u$ is the solution to~(\ref{e12}).
\end{proposition}
\begin{pf}
Set
\[
\psi_n:=(I-T_{n,\lambda})^{-1}f_n,\qquad
\psi:=(I-T_{\lambda})^{-1}f.\vadjust{\goodbreak}
\]
It is enough to show that
%
%
\begin{equation}
\label{e25b} \psi_n\to\psi\qquad\mbox{in } L^2(H,\mu).
\end{equation}
Let $\lambda\ge\lambda_0$, and write
\[
\psi-\psi_n=T_{\lambda}\psi-T_{n,\lambda}
\psi_n + f - f_n.
\]
Then, setting $\| \cdot\|_2 = \| \cdot\|_{L^2(\mu)}$,
\begin{eqnarray*}
\|\psi-\psi_n\|_{2}&\le&\|T_{n,\lambda}
\psi-T_{n,\lambda}\psi_n\|_{2} +\|T_{\lambda}
\psi-T_{n,\lambda}\psi\|_{2}+\|f-f_n\|_{2}
\\
&\le&\tfrac12 \|\psi-\psi_n\|_{2}+\|T_{\lambda}
\psi-T_{n,\lambda
}\psi\|_{2}+\|f-f_n\|_{2}.
\end{eqnarray*}
Consequently,
\[
\|\psi-\psi_n\|_{2}\le2\|T_{\lambda}
\psi-T_{n,\lambda}\psi\|_{2}+2\|f-f_n
\|_{2}.
\]
We also have
\[
\|T_{\lambda}\psi-T_{n,\lambda}\psi\|^2_{2} \le
\int_H\bigl|B(x)-B_n(x)\bigr|^2\bigl|D(\lambda-
L_2)^{-1} \psi(x)\bigr|^2\mu(dx).
\]
Therefore, by the dominate convergence theorem, it follows that
\[
{\lim_{n\to\infty}}\|T_{\lambda}\psi-T_{n,\lambda}\psi
\|_{2}=0.
\]
The conclusion follows.
\end{pf}

\subsection{Modified mild formulation}\label{sec2.3}

For any $i\in\mathbb N$ we denote the $i$th component of $B$ by
$B^{(i)}$, that is,
\[
B^{(i)}(x):=\bigl\langle B(x),e_i \bigr\rangle.
\]
Then for $\lambda\ge\lambda_0$ we consider the solution
$u^{(i)}$ of the equation
%
%
\begin{equation}
\label{e21} \lambda u^{(i)}-Lu^{(i)}-\bigl\langle B, D
u^{(i)} \bigr\rangle=B^{(i)},\qquad \mu\mbox{-a.s.}
\end{equation}
%
%
\begin{theorem}
\label{p4} Let $X_t$ be a mild solution of equation (\ref{SPDE})
on some filtered probability space,
let $u^{(i)}$ be the solution of (\ref{e21}) and set
$X^{(i)}_t=\langle X_t,e_i \rangle$. Then we have
%
%
\begin{eqnarray}
\label{e22} X_t^{(i)} &=& e^{-\lambda_i t}\bigl(\langle
x,e_i \rangle+u^{(i)}(x)\bigr) -u^{(i)}(X_t)
\nonumber
\\
&&{}+(\lambda+\lambda_i)\int_0^te^{-\lambda_i (t-s)}u^{(i)}(X_s)\,ds
\nonumber\\[-8pt]\\[-8pt]
&&{}+\int_0^te^{-\lambda_i (t-s)}\bigl(d\langle
W_s,e_i \rangle+\bigl\langle Du^{(i)}(X_s),dW_s
\bigr\rangle\bigr),\nonumber\\
&&\eqntext{t \ge0,\qquad \mathbb\P\mbox{-a.s.}}
\end{eqnarray}
\end{theorem}
\begin{pf} The proof uses in an essential way that, for any
$t>0$, $x \in H$,
the law $\pi_t(x,\cdot)$ of $X_t = X(t,x)$ is equivalent
to $\mu$. This follows from Theorem~\ref{tA1} (Girsanov's
theorem) in the \hyperref[app]{Appendix},\vadjust{\goodbreak} by which the law on $C([0,T]; H)$
of $X(\cdot, x)$ is equivalent to the law of the solution of
(\ref{SPDE}) with $B=0$, that is, it is equivalent to the law of the OU
process $Z(t,x)$ given in (\ref{ou11}).
In particular, their transition probabilities are equivalent. But it
is well known that the law of $Z(t,x)$ is equivalent to $\mu$ for
all $t >0 $ and $x \in H$ in our case; see~\cite{DZ}, Theorem 11.3.

Let us first describe a formal proof based on an heuristic use of
It\^o's formula, and then give the necessary rigorous details by
approximations.\vspace*{8pt}

\textit{Step} 1. Formal proof.

By It\^o's formula we have
\[
du^{(i)}(X_t)=\bigl\langle Du^{(i)}(X_t),dX_t
\bigr\rangle+\tfrac12 \operatorname{Tr} \bigl[D^2u^{(i)}(X_t)
\bigr] \,dt
\]
and so
\[
du^{(i)}(X_t)=Lu^{(i)}(X_t)\,dt +\bigl
\langle B(X_t),Du^{(i)} (X_t) \bigr\rangle\,dt +
\bigl\langle Du^{(i)}(X_t),dW_t \bigr\rangle.
\]
Now, using (\ref{e21}), we find that
%
%
\begin{equation}
\label{e28} du^{(i)}(X_t)= \lambda u^{(i)}(X_t)\,dt-B^{(i)}(X_t)\,dt+
\bigl\langle Du^{(i)}(X_t),dW_t \bigr\rangle.
\end{equation}
On the other hand, by (\ref{SPDE}) we deduce
\[
dX^{(i)}_t=-\lambda_iX^{(i)}_t\,dt+B^{(i)}(X_t)\,dt+dW_t^{(i)}.
\]
The expression for $B^{(i)}(X_t)$ that we get from this identity, we
insert into~(\ref{e28}). This yields
\[
dX^{(i)}_t=-\lambda_iX^{(i)}_t\,dt+
\lambda u^{(i)}(X_t)\,dt-du^{(i)}(X_t)+dW_t^{(i)}+
\bigl\langle Du^{(i)}(X_t),dW_t \bigr\rangle.
\]
By the variation of constants formula, this is equivalent to
\begin{eqnarray*}
X^{(i)}_t &=& e^{-\lambda_i t}\langle x,e_i
\rangle+\lambda\int_0^te^{-\lambda_i (t-s)}u^{(i)}(X_s)\,ds
\\
&&{}- \int_0^te^{-\lambda_i (t-s)}\,du^{(i)}(X_s)
+\int_0^te^{-\lambda_i (t-s)}
\bigl[dW_s^{(i)}+\bigl\langle Du^{(i)}(X_s),dW_s
\bigr\rangle\bigr].
\end{eqnarray*}
Finally, integrating by parts in the second integral yields (\ref
{e22}).\vspace*{8pt}

\textit{Step} 2. Approximation of $B$ and $u$.

Set
%
%
\begin{equation}
\label{e29} B_n(x)=\int_H B
\bigl(e^{A/n}x+y\bigr)N(0, Q_{1/n}) (dy),\qquad x\in H.
\end{equation}
Then $B_n$ is of $C^\infty$ class and all its derivatives are
bounded. Moreover, $\|B_n\|_0\le\|B\|_0$. It is easy to see that,
possibly passing to a subsequence,
%
%
\begin{equation}
\label{e30} B_n\to B,\qquad \mu\mbox{-a.s.} 
\end{equation}
[indeed $B_n\to B$ in $L^2(H,\mu;H)$; this result can
be first checked for continuous and bounded $B$].

Now we denote by $u_n^{(i)}$ the solution of the equation
%
%
\begin{equation}
\label{e31} \lambda u_n^{(i)}- Lu_n^{(i)}-
\bigl\langle B_n,Du_n^{(i)} \bigr
\rangle=B_n^{(i)},
\end{equation}
where $B_n^{(i)}=\langle B_n,e_i \rangle$. By Proposition~\ref{p3}
we have, possibly passing to a subsequence,
%
%
\begin{eqnarray}
\label{conv1}
\lim_{n\to\infty}u_n^{(i)}&=&u^{(i)},\qquad
\lim_{n\to\infty}Du_n^{(i)}=D u^{(i)},\qquad
\mu\mbox{-a.s.},
\nonumber\\[-8pt]\\[-8pt]
\sup_{n \ge1} \bigl\| u_n^{(i)} \bigr\|_{C^{1}_b(H)} &=&
C_{ i} < \infty,
\nonumber
\end{eqnarray}
where $u^{(i)}$ is the solution of (\ref{e21}).\vspace*{8pt}

\textit{Step} 3. Approximation of $X_t$.

For any $m\in\mathbb N$ we set $X_{m,t}:=\pi_m X_t$,
where $\pi_m=\sum_{j=1}^me_j\otimes e_j$. Then we have
%
%
\begin{equation}
\label{e32} X_{m,t}=\pi_mx+\int_0^tA_mX_{s}\,ds+
\int_0^t\pi_mB(X_{s})\,ds+
\pi_mW_t,
\end{equation}
where $A_m=\pi_mA$.

Now we denote by $u^{(i)}_{n,m}$ the solution of the equation
%
%
\begin{equation}
\label{e33} \lambda u^{(i)}_{n,m}- L u^{(i)}_{n,m}-
\bigl\langle\pi_m B_n \circ\pi_m,
Du^{(i)}_{n,m} \bigr\rangle=B_n^{(i)}
\circ\pi_m,
\end{equation}
where $(B_n \circ\pi_m) (x) = B_n ( \pi_m x)$, $x \in H$. Since only a
finite number of variables is involved, we have, equivalently,
\[
\lambda u^{(i)}_{n,m} - L_{m}u^{(i)}_{n,m}-
\bigl\langle\pi_m B_n \circ\pi_m,
Du^{(i)}_{n,m} \bigr\rangle=B_n^{(i)}
\circ\pi_m
\]
with
%
%
\begin{equation}
\label{e34} L_{m}\varphi=\tfrac12 \operatorname{Tr} \bigl[
\pi_mD^2\varphi\bigr]+\langle A_mx,D\varphi
\rangle.
\end{equation}
Moreover, since $u_{n,m}^{(i)}$ depends only on the first $m$
variables, we have
%
%
\begin{equation}
\label{uu} u_{n,m}^{(i)} (\pi_m y) =
u_{n,m}^{(i)} (y),\qquad y \in H, n,m,i \ge1.
\end{equation}
Applying a finite-dimensional It\^o formula to $u^{(i)}_{n,m}(X_{m,t})
= u^{(i)}_{n,m}(X_{t})$ yields
%
%
\begin{eqnarray}
\label{e35} %
du^{(i)}_{n,m}(X_{m,t})&=&
\tfrac12 \operatorname{Tr} \bigl[D^2u^{(i)}_{n,m}(X_{m,t})
\bigr]\,dt
\nonumber
\\
&&{}+\bigl\langle Du^{(i)}_{n,m}(X_{m,t}),
A_mX_t+\pi_mB(X_t) \bigr\rangle
\,dt
\\
&&{}+\bigl\langle Du^{(i)}_{n,m}(X_{m,t}),
\pi_md W_t \bigr\rangle.
\nonumber
\end{eqnarray}
On the other hand, by (\ref{e33}) we have
\begin{eqnarray*}
&&\lambda u^{(i)}_{n,m}(X_{m,t})-\tfrac12
\operatorname{Tr} \bigl[D^2u^{(i)}_{n,m}(X_{m,t})
\bigr]
\\
&&\quad{}-\bigl\langle Du^{(i)}_{n,m}(X_{m,t}),
A_mX_{m,t}+\pi_mB_n(X_{m,t})
\bigr\rangle\\
&&\qquad=B_n^{(i)}(X_{m,t}).
\end{eqnarray*}
Comparing with (\ref{e35})
yields
%
%
\begin{eqnarray}
\label{e36} du^{(i)}_{n,m}(X_{m,t}) &=&\lambda
u^{(i)}_{n,m}(X_{m,t})\,dt -B_n^{(i)}(X_{m,t})\,dt
\nonumber
\\
&&{}+\bigl\langle Du^{(i)}_{n,m}(X_{m,t}),
\pi_m\bigl(B(X_t)-B_n(X_{m,t})
\bigr) \bigr\rangle\,dt
\\
&&{}+\bigl\langle Du^{(i)}_{n,m}(X_{m,t}),
\pi_md W_t \bigr\rangle.
\nonumber
\end{eqnarray}
Taking into account (\ref{uu}), we rewrite (\ref{e36})
in the integral form as
%
%
\begin{eqnarray}
\label{si}
\quad&&u^{(i)}_{n,m}(X_{t}) -
u^{(i)}_{n,m}(X_{r})
\nonumber
\\
&&\qquad= \int_r^t \lambda
u^{(i)}_{n,m}(X_{s})\,ds - \int_r^t
B_n^{(i)}(X_{m,s})\,ds
\nonumber\\[-8pt]\\[-8pt]
&&\qquad\quad{}+ \int_r^t \bigl\langle Du^{(i)}_{n,m}(X_{s}),
\bigl(B(X_s)-B_n(X_{m,s})\bigr) \bigr\rangle
\,ds\nonumber\\
&&\qquad\quad{} + \int_r^t \bigl\langle
Du^{(i)}_{n,m}(X_{s}), dW_s \bigr
\rangle,
\nonumber
\end{eqnarray}
$t \ge r >0$. Let us fix $n$, $i \ge1$ and $x \in H$.

Possibly passing to a subsequence, and taking the limit in
probability (with respect to $\mathbb{P}$), from identity (\ref{si}),
we arrive at
%
%
\begin{eqnarray}
\label{e37} du_{n}^{(i)}(X_{t}) &=& \lambda
u_{n}^{(i)}(X_{t})\,dt-B_{n}^{(i)}%
(X_{t})\,dt
\nonumber\\
&&{}+\bigl\langle Du_{n}^{(i)}(X_{t}),
\bigl(B(X_{t})-B_{n}(X_{t})\bigr)\bigr\rangle\,dt\\
&&{}
+\bigl\langle Du_{n}^{(i)}(X_{t}),dW_t
\bigr\rangle,\qquad \mathbb{P}%
\mbox{-a.s.}
\nonumber
\end{eqnarray}
Let us justify such assertion.

First note that in equation (\ref{e33}) we have the drift term
$\pi_m B_n \circ\pi_m $ which converges pointwise
to $B_n$ and $B_n^{(i)} \circ\pi_m$ which
converges pointwise to $B_n^{(i)}$ as $m \to\infty$. Since such
functions are also uniformly bounded, we can apply Proposition
\ref{p3} and obtain that, possibly passing to a subsequence (recall
that $n$ is fixed),
%
%
\begin{eqnarray}
\label{conv12}
\lim_{m\to\infty}u_{n,m}^{(i)}&=&u^{(i)}_n,\qquad
\lim_{m\to
\infty}Du_{n,m}^{(i)}=D u^{(i)}_n,\qquad
\mu\mbox{-a.s.}, 
\nonumber\\[-8pt]\\[-8pt]
\sup_{m \ge1} \bigl\| u_{n,m}^{(i)} \bigr\|_{C^{1}_b(H)} &=&
C_{ i} < \infty.
\nonumber
\end{eqnarray}
Now we only consider the most involved terms in (\ref{si}).\vadjust{\goodbreak}

We have, using that the law $\pi_t(x, \cdot)$ of $X_t$ is absolutely
continuous with respect to $\mu$,
\begin{eqnarray*}
&&
\mathbb{E} \int_{r}^{t} \bigl|u_{n,m}^{(i)}(X_{s})
- u_{n}^{(i)}(X_{s})\bigr| \,ds\\
&&\qquad = \int
_{r}^{t} \,ds \int_H \bigl|
u^{(i)}_{n,m}(y)-u^{(i)}_{n}(y) \bigr|\,
\frac{d \pi_s(x,
\cdot)}{d \mu}(y) \mu(dy),
\end{eqnarray*}
which tends to 0,
as $m \to\infty$, by the dominated convergence
theorem [using~(\ref{conv12})].

This implies $\lim_{m\rightarrow\infty}\int_{r}^{t}\lambda
u_{n,m}^{(i)}(X_{s})\,ds =\int_{r}^{t}\lambda u_{n}^{(i)}(X_{s})\,ds $
in $L^{1}(\Omega, \mathbb{P})$.
Similarly,
we prove that $u^{(i)}_{n,m}(X_{t})$ and
$u^{(i)}_{n,m}(X_{r})$
converge, respectively, to $u^{(i)}_{n}(X_{t})$
and $u^{(i)}_{n}(X_{r})$ in $L^1$.

To show that
%
%
\begin{eqnarray}
\label{dop} &&\lim_{m \to\infty} \mathbb{E} \int_{r}^{t}
\bigl|\bigl\langle Du_{n,m}^{(i)}(X_{s}),
\pi_{m}\bigl(B(X_{s})-B_{n}(X_{m,s})
\bigr)\bigr\rangle
\nonumber\\[-8pt]\\[-8pt]
&&\hspace*{67pt}{}- \bigl\langle Du_{n}^{(i)}(X_{s}),
\bigl(B(X_{s})-B_{n}(X_{s})\bigr)\bigr\rangle
\bigr| \,ds =0,
\nonumber
\end{eqnarray}
it is enough to prove that $\lim_{m \to\infty} H_m + K_m =0$, where
\[
H_m = \mathbb{E} \int_{r}^{t} \bigl|\bigl
\langle Du_{n,m}^{(i)}(X_{s})- Du_{n}^{(i)}(X_{s}),
\pi_{m}\bigl(B(X_{s})-B_{n}(X_{m,s})
\bigr)\bigr\rangle\bigr| \,ds
\]
and
\[
K_m = \mathbb{E} \int_{r}^{t} \bigl|\bigl
\langle Du_{n}^{(i)}(X_{s}), \bigl[
\pi_{m}B(X_{s}) - B(X_{s})\bigr] +
\bigl[B_n(X_{s}) - \pi_m B_n(X_{m,s})
\bigr] \bigr\rangle\bigr| \,ds.
\]
It is easy to check that $\lim_{m \to\infty} K_m =0$. Let us deal
with $H_m$. We have
%
%
\begin{eqnarray}
\label{ft} H_m &\le& 2 \| B\|_0 \int
_{r}^{t} \mathbb{E} \bigl| Du_{n,m}^{(i)}(X_{s})-
Du_{n}^{(i)}(X_{s}) \bigr|\,ds
\nonumber\\[-8pt]\\[-8pt]
&\le&\int_{r}^{t} \,ds \int_H
\bigl| Du^{(i)}_{n,m}(y)-Du^{(i)}_{n}(y) \bigr|\,
\frac{d \pi_s(x, \cdot)}{d
\mu}(y) \mu(dy),
\nonumber
\end{eqnarray}
which tends to 0 as $m \to\infty$ by the dominated convergence
theorem [using~(\ref{conv12})].
This shows (\ref{dop}).

It remains to prove that
\[
\lim_{m \to\infty} \int_{r}^{t}\bigl\langle
Du_{n,m}^{(i)}(X_{s}), dW_s\bigr
\rangle=\int_{r}^{t}\bigl\langle
Du_{n}^{(i)}(X_{s}),dW_s\bigr
\rangle\qquad \mbox{in } L^{2}(\Omega,\mathbb{P}).
\]
For this purpose we use the isometry formula together
with
\[
\lim_{m\rightarrow\infty}\int_{r}^{t}\mathbb{E}\bigl
\llvert Du_{n,m} 
^{(i)}(X_{s})-Du_{n}^{(i)}(X_{s})
\bigr\rrvert^{2}\,ds =0
\]
[which can be proved arguing as in (\ref{ft})].
Thus we have proved (\ref{e37}).\vadjust{\goodbreak}

In order to pass to
the limit as $n \to\infty$ in (\ref{e37}), we recall formula
(\ref{conv1}) and argue as before [using also that $\pi_t(x, \cdot)
\ll\mu$]. We find
%
%
\begin{eqnarray}
\label{e38}
u^{(i)}(X_{t}) - u^{(i)}{(X_{r})}
&=& \int_r^t \lambda
u^{(i)}(X_{s})\,ds - \int_r^t
B^{(i)}(X_{s})\,ds \nonumber\\[-8pt]\\[-8pt]
&&{}+ \int_r^t
\bigl\langle Du^{(i)}(X_{s}),
dW_s \bigr\rangle,
\nonumber
\end{eqnarray}
$t \ge r >0$. Since $u$ is continuous and trajectories of $(X_t)$ are
continuous,
we can pass to the limit
as $r \to0^+$ in (\ref{e38}), $\mathbb{P}$-a.s., and obtain an
integral identity on $[0,t]$.

But
\[
dX_{t}^{(i)}=-\lambda_{i}X_{t}^{(i)}\,dt+B^{(i)}(X_{t})\,dt+dW_{t}^{(i)},\qquad
\mathbb{P}\mbox{-a.s.}
\]
Now we proceed as in step 1. Namely, we derive $B^{(i)}(X_{t})$ from
the identity above and insert in (\ref{e38}); this yields
\[
dX_{t}^{(i)}=-\lambda_{i}X_{t}^{(i)}\,dt+
\lambda u^{(i)}(X_{t})\,dt-du^{(i)}%
(X_{t})+dW_{t}^{(i)}+\bigl\langle
Du^{(i)}(X_{t}),dW_t\bigr\rangle,
\]
$\mathbb{P}$-a.s. Then we use the variation of constants formula.
\end{pf}
%
%
\begin{remark}
Formula (\ref{e38}) with $r=0$ seems to be of independent interest.
As an application, one can deduce, when $x \in H$ is
deterministic,
the representation formula
\[
{\mathbb E}\bigl[u^{(i)}(X_t)\bigr] = \int
_0^{\infty} e^{-\lambda t } {\mathbb E}
\bigl[B^{(i)}(X_t)\bigr] \,dt.
\]
This follows by taking the Laplace transform in both sides of
(\ref{e38}) (with $r=0$) and integrating by parts with respect to
$t$.
\end{remark}

The next lemma shows that $u(x) = \sum_{k \ge1} u^{(k)}(x) e_k$
[$u^{(k)}$ as in (\ref{e21})] is a well-defined function which
belongs to $C^1_b(H,H)$. Recall that $\lambda_0$ is defined in
(\ref{dt1}).
%
%
\begin{lemma} \label{de}
For $\lambda$
sufficiently large, that is, $\lambda\ge\widetilde\lambda$,
with $\widetilde\lambda= \widetilde\lambda(A, \| B\|_0)$,
there exists a unique $u = u_{\lambda} \in C^{1}_b(H,H)$ which
solves
\[
u(x) = \int_0^{\infty} e^{-\lambda t }
R_t \bigl(Du(\cdot) B(\cdot) + B(\cdot) \bigr) (x)\,dt,\qquad x \in H,
\]
where $R_t$ is the OU semigroup defined as in (\ref{e1}) and acting
on $H$-valued functions. Moreover, we have the following assertions:

\begin{longlist}
\item
Let $\epsilon\in[0, 1/2[$. Then, for any $h \in H$,
$(-A)^{\epsilon}Du(\cdot)[h]
\in C_b (H, H)$ and $\|(-A)^{\epsilon}Du(\cdot)[h] \|_0
$ $ \le C_{\epsilon, \lambda} |h|$;

\item for any $k \ge1$, $\langle u(\cdot), e_k\rangle= u^{(k)}$,
where $u^{(k)}$ is the solution defined in~(\ref{e21});\vadjust{\goodbreak}

\item there exists $c_{3}= c_3(A, \|B \|_0)>0$ such that, for any
$\lambda\ge\widetilde\lambda$, $u = u_{\lambda}$ satisfies
%
%
\begin{equation}
\label{grado} \| Du \|_0 \le\frac{c_3}{\sqrt{\lambda}}.
\end{equation}
\end{longlist}
\end{lemma}
\begin{pf} Let $E = C^1_b(H,H)$, and define the operator
$S_{\lambda}$,
\[
S_{\lambda} v (x)= \int_0^{\infty}
e^{-\lambda t } R_t \bigl(Dv(\cdot) B(\cdot) + B(\cdot) \bigr)
(x)\,dt,\qquad v \in E, x \in H.
\]
To prove that $S_{\lambda}\dvtx E \to E$, we take into account estimate
(\ref{e10b}) with $\epsilon=0$. Note that to check the Fr\'echet
differentiability of $S_{\lambda} v$ in each $x \in H$, we first show
its G\^ateaux differentiability. Then using formulas (\ref{e3})
and (\ref{e10b}), we obtain the continuity of
the G\^ateaux derivative from $H$ into $L(H)$. [$L(H)$ denotes
the Banach
space of all
bounded linear operators from $H$ into $H$ endowed with
$\| \cdot\|_{\mathcal L}$], and this
implies, in particular, the Fr\'echet
differentiability.

For $\lambda\ge(\lambda_0 \vee2 \| B\|_0) $,
$S_{\lambda}$ is a contraction and
so
there exists a unique $u \in E$
which solves $u = S_{\lambda} u$.
Using again (\ref{e10b}) we obtain (i). Moreover, (ii)~can be
deduced from the fact that, for each $k \ge1$,
$u_k = \langle u(\cdot), e_k \rangle$ is the unique solution to the equation
\[
u_k(x) = \int_0^{\infty}
e^{-\lambda t } R_t \bigl(\bigl\langle D u_k(\cdot), B(
\cdot) \bigr\rangle+ B^{k}(\cdot) \bigr) (x)\,dt,\qquad x \in H,
\]
in $C^1_b(H)$ (the uniqueness follows by the contraction principle)
and also the function $u^{(k)} \in C_b^1(H)$ given in (\ref{e21})
solves such equation. Finally (iii) follows easily from the estimate
\[
\| Du\|_0 \le\frac{C_{1,0}}{ \lambda^{1/2} } \bigl(\| Du\|_0 \| B
\|_0 + \| B\|_0 \bigr),\qquad \lambda\ge\bigl(\lambda_0
\vee\| B\|_0\bigr).
\]
\upqed
\end{pf}


\section{\texorpdfstring{Proof of Theorem \protect\ref{maintheorem}}{Proof of Theorem 1}}\label{sectionproof}
\label{sec3}

We start now the proof of pathwise uniqueness.

Let $X= (X_{t})$ and $Y = (Y_{t})$ be two continuous
${\mathcal F}_{t}$-adapted mild solutions (defined on the same filtered
probability space, solutions with respect to the same cylindrical
Wiener process), starting from the same $x $.

For the time being, $x$ is not specified (it may be also random,
${\mathcal F}_{0}$-measura\-ble). In the last part of the proof a
restriction on $x$ will emerge.

Let us fix $T>0$. Let $u = u_{\lambda}\dvtx H \to H$ be such that
$u(x) = \sum_{i \ge1} u^{(i)}(x)e_i$, $x \in H$, where
$u^{(i)} = u^{(i)}_{\lambda} $
solve (\ref{e21}) for some $\lambda$
large enough; see Proposition~\ref{p2}.

By (\ref{grado}) we may assume
that
$ \| Du \|_{0}\leq1/2$.
We have, for $t \in[0,T]$,
\begin{eqnarray*}
X_{t}-Y_{t}
&=& u ( Y_{t} ) -u (
X_{t} )\\
&&{} + ( \lambda-A ) \int_{0}^{t}e^{ ( t-s ) A}
\bigl( u ( X_{s} ) -u ( Y_{s} ) \bigr) \,ds
\\
&&{} + \int_{0}^{t}e^{ ( t-s ) A} \bigl( Du (
X_{s} ) -Du( Y_{s})\bigr) \,dW_{s}.
\end{eqnarray*}
It follows that
\begin{eqnarray*}
\llvert X_{t}-Y_{t}\rrvert&\leq&\frac{1}{2}\llvert
X_{t} -Y_{t}\rrvert+ \biggl| ( \lambda-A ) \int
_{0}^{t}e^{ ( t-s ) A} \bigl( u (
X_{s} ) -u ( Y_{s} ) \bigr) \,ds \biggr|
\\
&&{}+ \biggl| \int_{0}^{t}e^{ ( t-s ) A} \bigl( Du (
X_{s} ) -Du( Y_{s})\bigr) \,dW_{s} \biggr|.
\end{eqnarray*}
%
Let $\tau$ be a stopping time to be specified later.
Using that $1_{[0, \tau]}(t)
= 1_{[0, \tau]}(t) \cdot$ $ 1_{[0, \tau]}(s)$,
$0 \le s \le t \le T$,
we have (cf.~\cite{DZ}, page 187)
\begin{eqnarray*}
&&1_{[0, \tau]}(t) \llvert X_{t}-Y_{t}\rrvert\\
&&\qquad\leq C
1_{[0, \tau]}(t) \biggl| ( \lambda-A ) \int_{0}^{t}e^{ ( t-s ) A}
\bigl( u ( X_{s} ) -u ( Y_{s} ) \bigr) \,ds \biggr|
\\
&&\qquad\quad{} + C\biggl\llvert1_{[0, \tau]}(t) \int_{0}^{t}e^{ ( t-s )
A}
\bigl( Du ( X_{s} ) -Du ( Y_{s} ) \bigr)
1_{[0, \tau]}(s) \,dW_{s} 
\biggr\rrvert,
\end{eqnarray*}
where by $C$ we denote any constant which may depend on the
assumptions on $A$, $B$ and $T$.

Writing $1_{[0,\tau]}(s) X_{s}=\widetilde{X}_{s}$ and $1_{[0,\tau]}%
(s) Y_{s}=\widetilde{Y}_{s}$, and, using the Burkholder--Davis--Gundy
inequality with a large exponent $q>2$ which will be determined
below, we obtain (recall that $\Vert\cdot\Vert$ is the
Hilbert--Schmidt norm (cf.~\cite{DZ}, Chapter 4) with $C = C_q$),
\begin{eqnarray*}
&& {\mathbb E} \bigl[ \llvert\widetilde{X}_{t}-\widetilde{Y}_{t}
\rrvert^{q} \bigr] \\
&&\qquad\leq C {\mathbb E} \biggl[
e^{\lambda q t} \biggl\llvert( \lambda-A ) \int_{0}^{t}e^{ ( t-s ) A}
e^{- \lambda s} \bigl( u( X_{s}) -u( Y_{s}) \bigr)
1_{[0,\tau]}(s) \,ds \biggr\rrvert^q \biggr]
\\
&&\qquad\quad{} +C{\mathbb E} \biggl[ \biggl( \int_{0}^{t}1_{[0,\tau]}(s)
\bigl\llVert e^{ ( t-s ) A} \bigl( Du ( X_{s} ) -Du (
Y_{s} ) \bigr) \bigr\rrVert^{2}\,ds
\biggr)^{q/2} \biggr].
\end{eqnarray*}
In the sequel we introduce a parameter $\theta>0$, and $C_{\theta}$
will denote suitable constants such that $C_{\theta}\rightarrow0$ as
$\theta\rightarrow+\infty$ (the constants may change from line to
line). This idea of introducing $\theta$ and $C_{\theta}$ is
suggested by~\cite{issoglio}, page 8. Similarly, we will indicate
by $C(\lambda)$ suitable constants
such that $C(\lambda)\rightarrow0$ as
$\lambda\rightarrow+\infty$.

From the previous inequality we deduce, multiplying by $ e^{-q\theta
t}$, for any $\theta>0$,
%
%
\begin{eqnarray}
\label{serve}
&& {\mathbb E} \bigl[ e^{-q\theta t}\llvert\widetilde
{X}_{t}-\widetilde{Y}%
_{t}\rrvert^{q}
\bigr]
\nonumber
\\
&&\qquad \leq C{\mathbb E} \biggl[ \biggl\llvert( \lambda-A ) \int
_{0}^{t}%
e^{-\theta( t-s ) }e^{ ( t-s ) A}
\bigl( u ( X_{s} ) -u ( Y_{s} ) \bigr)\nonumber\\
&&\hspace*{167pt}{}\times
e^{-\theta
s}1_{ [
0,\tau] } ( s ) \,ds\biggr\rrvert^{q} \biggr]
\\
&&\qquad\quad{} +C{\mathbb E} \biggl[ \biggl( \int_{0}^{t}e^{-2\theta( t-s ) }
\bigl\llVert e^{ (
t-s ) A} \bigl( Du ( X_{s} ) -Du (
Y_{s} ) \bigr) \bigr\rrVert^{2}\nonumber\\
&&\qquad\quad\hspace*{149pt}{}\times e^{-2\theta s}1_{ [ 0,\tau]}
( s ) \,ds \biggr)^{q/2} \biggr].
\nonumber
\end{eqnarray}
Let us deal with the first term in the right-hand side. Integrating
over $ [ 0,T ] $, and assuming $\theta\geq\lambda$, we
get
\begin{eqnarray*}
&& \int_{0}^{T}C \mathbb{E} \biggl[ \biggl
\llvert( \lambda-A ) \int_{0}%
^{t}e^{-\theta( t-s ) }e^{ ( t-s ) A}
\bigl( u ( X_{s} ) -u ( Y_{s} ) \bigr)
e^{-\theta
s}1_{ [
0,\tau] } ( s ) \,ds\biggr\rrvert^{q} \biggr] \,dt
\\
&&\qquad =C \mathbb{E} \biggl[ \int_{0}^{T}\biggl
\llvert( \lambda-A ) \int_{0}%
^{t}e^{-\theta( t-s ) }e^{ ( t-s ) A}
\bigl( u ( X_{s} ) -u ( Y_{s} ) \bigr)\\
&&\qquad\quad\hspace*{150pt}{}\times
e^{-\theta
s}1_{ [
0,\tau] } ( s ) \,ds\biggr\rrvert^{q}\,dt \biggr]
\\
&&\qquad \leq I_{1}+I_{2},
\end{eqnarray*}
where
\begin{eqnarray*}
I_{1}&=&C 2^{q-1} \mathbb{E} \biggl[ \int
_{0}^{T}\biggl\llvert( \theta-A ) \int
_{0}%
^{t}e^{ ( t-s ) ( A-\theta) } \bigl( u (
X_{s} ) -u ( Y_{s} ) \bigr) e^{-\theta
s}1_{ [ 0,\tau] }
( s ) \,ds\biggr\rrvert^{q}\,dt \biggr],
\\
I_{2}&=&C \mathbb{E} \biggl[ \int_{0}^{T}
\biggl\llvert\int_{0}^{t}2\theta
e^{-\theta( t-s ) }e^{ ( t-s ) A} \bigl( u ( X_{s} ) -u (
Y_{s} ) \bigr) e^{-\theta
s}1_{ [ 0,\tau] } ( s ) \,ds\biggr
\rrvert^{q}\,dt \biggr],
\end{eqnarray*}
%
Let us estimate $I_1$ and $I_2$ separately. To estimate $I_1$, we
use the $L^q$-maximal inequality; see, for instance,
\cite{dap}, Section 1. This implies that, $\mathbb P$-a.s.,
\begin{eqnarray*}
&& \int_{0}^{T}\biggl\llvert( \theta-A ) \int
_{0}%
^{t}e^{ ( t-s ) ( A-\theta) } \bigl( u (
X_{s} ) -u ( Y_{s} ) \bigr) e^{-\theta
s}1_{ [ 0,\tau] }
( s ) \,ds\biggr\rrvert^{q}\,dt
\\
&&\qquad\le C_4 \int_0^T
e^{-\theta q s} \bigl\llvert u( X_{s}) -u( Y_{s}) \bigr
\rrvert^q 1_{ [ 0,\tau] } ( s ) \,ds,
\end{eqnarray*}
where it is important to remark that $C_4$ is independent on $\theta
>0$. To see this, look at~\cite{dap}, Theorem 1.6, page 74, and note
that for a fixed $\alpha\in(\pi/2, \pi)$, there exists $c=
c(\alpha)$ such that for any $\theta>0$, $\mu\in{\mathbb C}$, $\mu
\not=0$, such that
$|{\operatorname{arg}}(\mu)| < \alpha$, we have
%
%
\begin{equation}
\label{mavv} \bigl\| \bigl(\mu- (A - \theta)\bigr)^{-1} \bigr\|_{\mathcal L}
\le\frac{c(\alpha)}{|\mu|}.
\end{equation}
Continuing we get
\[
I_1 \le C(\lambda) \int_0^T
e^{-\theta q s} \llvert\widetilde X_{s} - \widetilde Y_{s}
\rrvert^q \,ds
\]
with $C(\lambda) = C_0 \| Du \|_0^{q} \to0$ as $\lambda\to+
\infty$.\vadjust{\goodbreak}

Let us deal with the term $I_{2}$. Given $t\in(0,T]$, the function
$s\mapsto\theta$ $ e^{-\theta( t-s ) }( 1-e^{-\theta t})^{-1}$ is a
probability density on $ [ 0,t ] $, and thus,
by Jensen's inequality,
\begin{eqnarray*}
I_{2}&=&C2^{q} \mathbb{E} \biggl[ \int
_{0}^{T} \bigl( 1-e^{-\theta
t}
\bigr)^{q}
\\
&&\hspace*{46pt}{}\times\biggl|\int_{0}^{t}e^{ ( t-s ) A}
\bigl( u ( X_{s} ) -u ( Y_{s} ) \bigr)
e^{-\theta s}1_{ [ 0,\tau]
} ( s ) \frac{\theta e^{-\theta( t-s )
}}{1-e^{-\theta t}}\,ds
\biggr|^{q}\,dt \biggr]
\\
& \leq&\widetilde C \mathbb{E} \biggl[ \int_{0}^{T}
\bigl( 1-e^{-\theta
t} \bigr)^{q}\int_{0}%
^{t}
\bigl\llvert u ( X_{s} ) -u ( Y_{s} ) \bigr
\rrvert^{q}e^{-q\theta s}1_{ [ 0,\tau] } ( s )
\frac{\theta
e^{-\theta( t-s ) }}{1-e^{-\theta t}}\,ds\,dt \biggr]
\\
& \leq& \widetilde C\llVert Du\rrVert_{0}^{q} \mathbb{E}
\biggl[ \int_{0}^{T} \bigl( 1-e^{-\theta t}
\bigr)^{q-1}\int_{0}^{t}\theta
e^{-\theta( t-s ) }\llvert\widetilde{X}_{s}-\widetilde{Y}_{s}
\rrvert^{q}e^{-q\theta
s}\,ds\,dt \biggr]
\\
& = & \widetilde C\llVert Du\rrVert_{0}^{q} \mathbb{E} \biggl[
\int_{0}^{T} \biggl( \int_{s}%
^{T}
\bigl( 1-e^{-\theta t} \bigr)^{q-1}\theta e^{-\theta(
t-s ) }\,dt \biggr)
\llvert\widetilde{X}_{s}-\widetilde{Y}_{s}
\rrvert^{q}e^{-q\theta s}\,ds \biggr]
\\
& \leq &\widetilde C(\lambda) \mathbb{E} \biggl[ \int_{0}^{T}
\llvert\widetilde{X}_{s}-\widetilde{Y}_{s}
\rrvert^{q}e^{-q\theta
s}\,ds \biggr],
\end{eqnarray*}
because $\int_{s}^{T} ( 1-e^{-\theta t} )^{q-1}\theta
e^{-\theta( t-s ) }\,dt\leq1$, for any $\theta\ge\lambda
$. Thus we have found
%
%
\begin{eqnarray}
\label{ritorna} && {\mathbb E} \biggl[ \biggl\llvert( \lambda-A ) \int
_{0}^{t}%
e^{-\theta( t-s ) }e^{ ( t-s ) A}
\bigl( u ( X_{s} ) -u ( Y_{s} ) \bigr)
e^{-\theta
s}1_{ [ 0,\tau] } ( s ) \,ds\biggr\rrvert^{q} \biggr]
\nonumber\\[-8pt]\\[-8pt]
&&\qquad\le C(\lambda) \mathbb{E} \biggl[ \int_{0}^{T}
\llvert\widetilde{X}_{s}-\widetilde{Y}_{s}
\rrvert^{q}e^{-q\theta
s}\,ds \biggr].
\nonumber
\end{eqnarray}
%
Now let us estimate the second term on the right-hand side of
(\ref{serve}). For
$t > 0$ fixed,
Lemma~\ref{da} from Appendix~\ref{aapA.2} implies that $ds \otimes
\P$-a.s. on $[0,t] \times\Omega$
\begin{eqnarray*}
&& \bigl\llVert e^{ ( t-s ) A} \bigl( Du ( X_{s} ) -Du (
Y_{s} ) \bigr) \bigr\rrVert^{2}\\
&&\qquad =\sum
_{n\geq
1}e^{-2\lambda_{n}(t-s)} \bigl|Du^{(n)}(X_{s})-Du^{(n)}(Y_{s})\bigr|^{2}
\\
&&\qquad =\sum_{k \ge1}\sum_{n\geq
1}e^{-2\lambda_{n}(t-s)}
\bigl| D_k u^{(n)}(X_{s})-D_ku^{(n)}(Y_{s})\bigr|^{2}
\\
&&\qquad = \sum_{k, n \ge1} e^{-2\lambda_{n}(t-s)} \biggl| \int
_{0}^{1} \bigl\langle D D_k
u^{(n)}\bigl(Z_{s}^{r}\bigr), X_s -
Y_s \bigr\rangle\,dr \biggr|^2
\\
&&\qquad \leq\sum_{n\geq1}e^{-2\lambda_{n}(t-s)} \biggl( \int
_{0}^{1}\bigl\| D^{2}u^{(n)}
\bigl(Z_{s}^{r}\bigr)\bigr\|^{2}\,dr \biggr)
|X_{s}-Y_{s}|^{2}
\\
&&\qquad =\int_{0}^{1} \biggl( \sum
_{n\geq1}e^{-2\lambda_{n}(t-s)}\bigl\| D^{2}%
u^{(n)}\bigl(Z_{s}^{r}\bigr)\bigr\|^{2}
\biggr) \,dr\, |X_{s}-Y_{s}|^{2},
\end{eqnarray*}
where $D_k u^{(n)} = \langle D u^{(n)}, e_k\rangle$,
$D_h D_k
u^{(n)} = \langle D^2 u^{(n)} e_h, e_k\rangle$
and $\|
D^{2}u^{(n)}(z)\|^{2} = \sum_{h, k \ge1} |D_h D_k
u^{(n)}(z)|^2$, for $\mu$-a.e. $z \in H$, and as before,
%
\[
Z_{t}^{r}=Z_{t}^{r,x}=rX_{t}+(1-r)Y_{t}.
\]
Integrating the second term in (\ref{serve}) in $t$ over $[0,T]$, we
thus find
\begin{eqnarray*}
\Gamma_{T}:\!&=&\int_{0}^{T}{\mathbb
E} \biggl[ \biggl( \int_{0}^{t}e^{-2\theta(
t-s) }1_{[0,\tau]}(s)e^{-2\theta s}
\\[-1pt]
&&\hspace*{49pt}{} \times\bigl\| e^{( t-s) A} \bigl( Du ( X_{s} ) -Du (
Y_{s} ) \bigr) \bigr\|^{2}\,ds
\biggr)^{q/2} \biggr] \,dt
\\[-1pt]
&\leq&\int_{0}^{T}{\mathbb E} \biggl[ \biggl(
\int_{0}^{t}e^{-2\theta( t-s)
}1_{[0,\tau]}(s)
\\[-1pt]
&&\hspace*{49pt}{} \times\int_{0}^{1} \biggl( \sum
_{n\geq1}e^{-2\lambda_{n}(t-s)}\bigl\| D^{2}u^{(n)}
\bigl(Z_{s}^{r}\bigr)\bigr\|^{2} \biggr) \,dr\\[-1pt]
&&\hspace*{135.4pt}{}\times e^{-2\theta s}|X_{s}-Y_{s} 
|^{2}\,ds \biggr)^{q/2} \biggr] \,dt.
\end{eqnarray*}
Now we consider $\delta\in(0,1)$ such that $(-A)^{-1+\delta}$ is
of finite trace. Then
\begin{eqnarray*}
\Gamma_{T} &\le&\int_{0}^{T}{\mathbb
E} \biggl[ \biggl( \int_{0}^{t}
\frac{e^{-2\theta( t-s) }}{(t-s)^{{1 -
\delta}}}1_{[0,\tau]}(s)
\\[-1pt]
&&\hspace*{50pt}{} \times\int_{0}^{1} \biggl( \sum
_{n\geq1}\bigl( \lambda_{n}(t-s)\bigr)^{{1 - \delta}}\\[-1pt]
&&\hspace*{98pt}{}\times { e^{-2\lambda_{n}(t-s)} } \frac{ \| D^{2}u^{(n)}(Z_{s}^{r})
\|^{2}} {{\lambda_{n}^{{1 - \delta}}}} \biggr) \,dr\\[-1pt]
&&\hspace*{149.6pt}{}\times e^{-2\theta
s}|X_{s}-Y_{s}|^{2}\,ds
\biggr)^{q/2} \biggr] \,dt
\\[-1pt]
&\leq& C\int_{0}^{T}{\mathbb E} \biggl[ \biggl(
\int_{0}^{t}\frac
{e^{-2\theta(
t-s) }}{(t-s)^{{1 - \delta}}}1_{[0,\tau]}(s)
\\[-1pt]
&&\hspace*{59.7pt}{} \times\int_{0}^{1} \biggl( \sum
_{n\geq1}\frac{1}{\lambda_{n}^{{1 - \delta}}}\bigl\| D^{2}u^{(n)}
\bigl(Z_{s}^{r}\bigr)\bigr\|^{2} \biggr) \,dr\\[-1pt]
&&\hspace*{125pt}{}\times e^{-2\theta
s}|X_{s}-Y_{s}|^{2}\,ds
\biggr)^{q/2} \biggr] \,dt.
\end{eqnarray*}
Let us explain the motivation of the previous estimates: on the one
side we isolate the term $\frac{e^{-2\theta( t-s )
}}{(t-s)^{{1 - \delta}}}$ which
will produce a constant $C_{\theta}$ arbitrarily small for large
$\theta$; on the
other side, we keep the term\vadjust{\goodbreak}
$\frac{1}{\lambda_{n}^{{1 - \delta}}}$ in the series $\sum_{n\geq
1}\frac{1}%
{\lambda_{n}^{{1 - \delta}}}\|
D^{2}u^{(n)}(Z_{s}^{r})\|^{2}$; otherwise, later on (in the
next proposition), we could not evaluate high powers of this
series.

Using the (triple) H\"{o}lder inequality in the integral with
respect to $s$, with
$\frac{2}{q}+\frac{1}{\beta}+\frac{1}{\gamma}=1$, $\gamma>1$ and
$\beta>1$ such that ${(1 - \delta)}\beta<1$, and Jensen's inequality
in the integral with respect to $r$, we find
%
%
\begin{equation}
\Gamma_{T}\leq\widetilde C_{\theta}{\mathbb E} \biggl[
\Lambda_{T}\int_{0}^{T}e^{-q\theta s}%
| \widetilde X_{s}- \widetilde Y_{s}|^{q}\,ds \biggr],
\end{equation}
where
\[
\widetilde C_{\theta}= \biggl( \int_{0}^{T}
\frac{e^{-2\beta\theta
r}}{r^{{(1 - \delta)}\beta}%
}\,dr \biggr)^{q/2\beta}%
\]
(which converges to zero as $\theta\rightarrow\infty$) and
\[
\Lambda_{T}:=\int_{0}^{T} \biggl(
\int_{0}^{t}1_{[0,\tau]}(s) \int
_{0}%
^{1} \biggl( \sum
_{n\geq1}\frac{1}{\lambda_{n}^{{1 - \delta}}}\bigl\| D^{2}%
u^{(n)}\bigl(Z_{s}^{r}\bigr)\bigr\|^{2}
\biggr)^{\gamma} \,dr \,ds \biggr)^{q/2\gamma}\,dt.
\]
We may choose $\gamma=\frac{q}{2}$ so that $\frac{q}{2\gamma}=1$.
This is compatible with the other constraints, namely $q>2$,
$ \frac{2}{q}+\frac{1}{\beta}+\frac{1}{\gamma}=1$, $\beta>1$ such
that ${(1 - \delta)}\beta<1$, because we may choose $\beta>1$
arbitrarily close to 1 and then solve
$\frac{4}{q}+\frac{1}{\beta}=1$ for $q$, which would require $q>4$.
So, from now on we fix $q \in(4, \infty)$ and $\gamma= q/2$. Hence
\begin{eqnarray*}
\Lambda_{T} :\!&=&\int_{0}^{T}\int
_{0}^{t}1_{[0,\tau]}(s) \int
_{0}^{1} \biggl( \sum
_{n\geq1}\frac{1}{\lambda_{n}^{{1 - \delta}}}\bigl\| D^{2}u^{(n)}
\bigl(Z_{s}^{r}%
\bigr)\bigr\|^{2}
\biggr)^{\gamma} \,dr \,ds \,dt
\\
&\leq& T\cdot\int_{0}^{T\wedge\tau} \int
_{0}^{1} \biggl( \sum
_{n\geq1}\frac{1}{\lambda_{n}^{{1 - \delta}}}\bigl\| D^{2}u^{(n)}
\bigl(Z_{s}^{r}\bigr)\bigr\|^{2}
\biggr)^{\gamma} \,dr \,ds.
\end{eqnarray*}
%
Define now, for any $R>0$, the stopping time
\[
\tau_{R}^{x}=\inf\biggl\{ t\in[ 0,T ]\dvtx\int
_{0}^{t} \int_{0}%
^{1}
\biggl( \sum_{n\geq1}\frac{1}{\lambda_{n}^{{1 - \delta}}}\bigl\|
D^{2}%
u^{(n)}\bigl(Z_{s}^{r}
\bigr)\bigr\|^{2} \biggr)^{\gamma} \,dr \,ds\geq R \biggr\}
\]
and $\tau_{R}^{x}=T$ if this set is empty. Take $\tau=\tau_{R}^{x}$
in the previous expressions and collect the previous estimates.
Using also (\ref{ritorna}) we get from (\ref{serve}),
for
any $\theta\ge\lambda$,
\begin{eqnarray*}
&& \int_{0}^{T}e^{-q\theta t}{\mathbb E}|
\widetilde{X}_{t}-\widetilde{Y}_{t}|^{q}\,dt
\\
&&\qquad \leq C(\lambda) \int_{0}^{T}e^{-q\theta s}{
\mathbb E}|\widetilde{X}_{s}-\widetilde{Y}%
_{s}|^{q}\,ds\\
&&\qquad\quad{}+
\widetilde C_{\theta}R\int_{0}^{T}e^{-q\theta s}{
\mathbb E}|\widetilde{X}_{s}-\widetilde{Y}_{s}|^{q}\,ds.
\end{eqnarray*}
Now we fix $\lambda$ large enough such that $ C (\lambda) <1$
and consider $\theta$ greater of such~$\lambda$.

For sufficiently large $\theta=\theta_{R}$, \textit{depending on}
$R$,
\[
{\mathbb E} \biggl[\int_{0}^{T}e^{-q\theta_{R} t}1_{[0,\tau_{R}]}(t)
|X_{t}-Y_{t}%
|^{q}\,dt \biggr]={\mathbb
E} \biggl[\int_{0}^{\tau_{R}}e^{-q\theta_{R}
t}
|X_{t}-Y_{t}%
|^{q}\,dt \biggr]=0.
\]
In other words, for every $R>0$, $\mathbb{P}$-a.s., $X=Y$ on
$ [ 0,\tau_{R} ] $ (identically in $t$, since $X$ and
$Y$ are continuous processes). We have
$\lim_{R\rightarrow\infty}\tau_{R}=T$, $\mathbb{P}$-a.s., because
of the next proposition. Hence, $\mathbb{P}$-a.s., $X=Y$ on
$ [ 0,T ] $, and the proof is complete.
%
%
\begin{proposition}
\label{stop} For $\mu$-a.e. $x\in H$, we have
$\mathbb{P} (S_{T}^x <\infty)=1$, where
%
\[
S_{T}^{x}=\int_{0}^{T}
\int_{0}^{1} \biggl( \sum
_{n\geq1}\frac
{1}{\lambda_{n}^{{1 - \delta}}}\bigl\| D^{2}u^{(n)}
\bigl(Z_{s}^{r}\bigr)\bigr\|^{2}
\biggr)^{\gamma} \,dr \,ds
\]
with $\gamma= q/2$.
The result is true also for a random ${\mathcal F}_{0}$-measurable,
$H$-valued initial
condition under the assumptions stated in Theorem~\ref{maintheorem}.
\end{proposition}
\begin{pf}
We will show that, for any $x\in H$, $\mu$-a.s.,
\[
{\mathbb E}\bigl[S_{T}^{x}\bigr]<+\infty.
\]
We will also show this result for random initial conditions under
the specified assumptions.\vspace*{8pt}

\textit{Step} 1. In this step $x\in H$ is given, without restriction.
Moreover, the result is true for a general ${\mathcal F}_{0}
$-measurable initial condition $x$ without restrictions on its
law.

We have%
\[
Z_{t}^{r}=e^{tA}x+\int_{0}^{t}e^{ ( t-s ) A}
\bar{B}_{s}^{r} 
\,ds+\int_{0}^{t}e^{ ( t-s ) A}\,dW_{s},
\]
where%
\[
\bar{B}_{s}^{r}=\bigl[rB(X_{s})+(1-r)B(Y_{s})
\bigr],\qquad r\in[0,1].
\]
Define
\[
\rho_{r}=\exp\biggl( -\int_{0}^{t}
\bar{B}_{s}^{r}\,dW_{s}-\frac{1}{2}\int
_{0}^{t}\bigl|\bar{B}_{s}^{r}\bigr|^{2}\,ds
\biggr).
\]
We have, since $|\bar{B}_{s}^{r}|\leq\| B\|_{0}$,
%
%
\begin{equation}
\label{estimateeasy} {\mathbb E} \biggl[ \exp\biggl( k\int_{0}^{T}\bigl|
\bar{B}_{s}^{r}\bigr|^{2}\,ds \biggr) \biggr] \leq
C_{k}<\infty
\end{equation}
for all $k\in\mathbb{R}$, independently of $x$ and $r$, simply
because $B$ is bounded. Hence an infinite-dimensional version of
Girsanov's theorem with respect to a cylindrical Wiener process (the
proof of which is included in the \hyperref[app]{Appendix};
see Theorem~\ref{tA1})
applies and gives us that
\[
\widetilde W_{t}:=W_{t}+\int_{0}^{t}
\bar{B}_{s}^{r}\,ds
\]
is a cylindrical Wiener process on $ ( \Omega, {\mathcal F},
( {\mathcal F}_{t} )_{t\in[ 0,T ]
},\widetilde{\mathbb{P}}_{r} ) $ where
$\frac{d\widetilde{\mathbb{P}}_{r}}{d\mathbb{P}}|_{{\mathcal F}_{T}%
}=\rho_{r}$. Hence
\[
Z_{t}^{r}=e^{tA}x+\int_{0}^{t}e^{ ( t-s ) A}\,d
\widetilde W_{s}%
\]
is the sum of a stochastic integral which is Gaussian with respect
to
$\widetilde{\mathbb{P}}_{r}$, plus the independent (because
${\mathcal F}_{0}%
$-measurable) random variable $e^{tA}x$. Its law is uniquely
determined by $A$, $r$ and the law of $x$.

Denote by $W_{A} ( t ) $ the
process%
\[
W_{A} ( t ):=\int_{0}^{t}e^{ ( t-s )
A}\,dW_{s}.
\]
We have $e^{\cdot A}x+W_{A} ( \cdot) =Z^{r}$ in law.
We have
%
%
\begin{eqnarray}
\label{ps} {\mathbb E}\bigl[S_{T}^{x}\bigr]&=&{\mathbb E}
\biggl[ \int_{0}^{T} \int_{0}^{1}
\biggl( \sum_{n\geq1}%
\frac{1}{\lambda_{n}^{{1 - \delta}}}
\bigl\| D^{2}u^{(n)}\bigl(Z_{s}^{r}\bigr)
\bigr\|^{2} \biggr)^{\gamma} \,dr \,ds \biggr]
\nonumber\\[-8pt]\\[-8pt]
&=&\int_{0}^{T}\int_{0}^{1}{
\mathbb E} \biggl[ \biggl( \sum_{n\geq
1}
\frac{1}%
{\lambda_{n}^{{1 - \delta}}}\bigl\|
D^{2}u^{(n)}\bigl(Z_{s}^{r}\bigr)
\bigr\|^{2} \biggr)^{\gamma} \biggr] \,dr \,ds.
\nonumber
\end{eqnarray}
Applying the Girsanov theorem, we find, for $r\in[0,1]$,
\begin{eqnarray*}
&& {\mathbb E} \biggl[\rho_{r}^{-1/2}\rho_{r}^{1/2}
\biggl( \sum_{n\geq1}\frac{1}%
{
\lambda_{n}^{{1 - \delta}}}\bigl\| D^{2}u^{(n)}
\bigl(Z_{s}^{r}\bigr)\bigr\|^{2}
\biggr)^{\gamma
} \biggr]
\\
&&\qquad \leq\bigl({\mathbb E}\bigl[\rho_{r}^{-1}\bigr]
\bigr)^{1/2} \biggl({\mathbb E} \biggl[\rho_{r} \biggl( \sum
_{n\geq1}\frac{1}{\lambda_{n}^{{1 -
\delta}}}\bigl\| D^{2}u^{(n)}
\bigl(Z_{s}^{r}\bigr)\bigr\|^{2}
\biggr)^{2\gamma} \biggr] \biggr)^{1/2}
\\
&&\qquad \leq{\mathbb E}\bigl[\rho_{r}^{-1}\bigr]+{\mathbb E}
\biggl[ \biggl( \sum_{n\geq1}\frac{1}{\lambda_{n}^{{1 - \delta}}}\bigl\|
D^{2}u^{(n)}\bigl(e^{sA}x+W_{A}(s)
\bigr)\bigr\|^{2} \biggr)^{2\gamma} \biggr].
\end{eqnarray*}
By (\ref{ps}) it follows that
%
%
\begin{eqnarray}
\label{ftt} {\mathbb E}\bigl[S_{T}^{x}\bigr]&\leq& T\int
_{0}^{1}{\mathbb E}\bigl[\rho_{r}^{-1}
\bigr]\,dr\nonumber\\[-8pt]\\[-8pt]
&&{}+\int_{0}^{T}{\mathbb E} \biggl[ \biggl(
\sum_{n\geq1}\frac{1}{\lambda_{n}^{{1 - \delta}}}\bigl\|
D^{2}u^{(n)}%
\bigl(e^{sA}x+W_{A}(s)
\bigr)\bigr\|^{2} \biggr)^{2\gamma} \biggr] \,ds.\nonumber
\end{eqnarray}

\textit{Step} 2. We have ${\mathbb E} [ \rho_{r}^{-1} ]
\leq C<\infty$ independently of $x\in H$
(also in the case of an ${\mathcal F}_{0}$-measurable $x$) and $r\in
[0,1]$.
Indeed,
\begin{eqnarray*}
{\mathbb E} \bigl[ \rho_{r}^{-1} \bigr] &=& {\mathbb E}
\biggl[ \exp\biggl( \int_{0}^{t}\bar
{B}_{s}^{r}\,dW_{s}+\frac{1}{2}\int
_{0}^{t}\bigl|\bar{B}_{s}^{r}\bigr|^{2}\,ds
\biggr) \biggr]
\\
&=& {\mathbb E} \biggl[ \exp\biggl( \int_{0}^{t}
\bar{B}_{s}^{r}\,dW_{s}+ \biggl(
\frac{3}%
{2}-1 \biggr) \int_{0}^{t}\bigl|
\bar{B}_{s}^{r}\bigr|^{2}\,ds \biggr) \biggr]
\\
&\leq& {\mathbb E} \biggl[ \exp\biggl( \int_{0}^{t}2
\bar{B}_{s}^{r}\,dW_{s}-\frac{1}{2}\int_{0}^{t}\bigl|2\bar{B}_{s}^{r}\bigr|^{2}\,ds
\biggr) \biggr]^{1/2}C_{3}^{1/2}%
\end{eqnarray*}
by (\ref{estimateeasy}). But
\[
{\mathbb E} \biggl[ \exp\biggl( \int_{0}^{t}2
\bar{B}_{s}^{r}\,dW_{s}-\frac{1}{2}\int
_{0}^{t}\bigl|2\bar{B}_{s}^{r}\bigr|^{2}\,ds
\biggr) \biggr] =1,
\]
because of Girsanov's theorem. Therefore, ${\mathbb E} [
\rho_{r}^{-1} ] $ is bounded uniformly in $x$ and~$r$.\vspace*{8pt}

\textit{Step} 3. Let us come back to (\ref{ftt}). To prove that
${\mathbb E}[S_{T}%
^{x}]<+\infty$ and hence finish the proof, it is enough to verify that%
%
%
\begin{equation}
\label{df} \int_{0}^{T}{\mathbb E} \biggl[
\biggl( \sum_{n\geq1}\frac
{1}{\lambda_{n}^{{1 - \delta}}%
}\bigl\|
D^{2}u^{(n)}\bigl(e^{sA}x+W_{A}(s)
\bigr)\bigr\|^{2} \biggr)^{2\gamma} \biggr] \,ds<\infty.
\end{equation}
If $\mu_{s}^{x}$ denotes the law of $e^{sA}x+W_{A} ( s )
$, we
have to prove that%
%
%
\begin{equation}
\label{cruciale} \int_{0}^{T}\int
_{H} \biggl( \sum_{n\geq1}
\frac{1}{\lambda_{n}^{{1 -
\delta}}}\bigl\| D^{2}u^{(n)}(y)\bigr\|^{2}
\biggr)^{2\gamma}\mu_{s}^{x}(dy)\,ds<\infty.
\end{equation}
Now we check (\ref{df}) for deterministic $x \in H$. In step 4
below, we will consider the case where $x$ is an ${\mathcal
F}_0$-measurable r.v.

We estimate
\[
\biggl( \sum_{n\geq1}\frac{1}{\lambda_{n}^{{1 - \delta}}}\bigl\|
D^{2}u^{(n)}%
(y)\bigr\|^{2}
\biggr)^{2\gamma}\leq\biggl( \sum_{n\geq1}
\frac{1}{\lambda_{n}^{{(1 - \delta)} (
{2\gamma}/({2\gamma-1}) ) }} \biggr)^{2\gamma
-1}\sum_{n\geq1}
\bigl\| D^{2}u^{(n)}(y)\bigr\|^{4\gamma}.
\]
Since $\frac{2\gamma}{2\gamma-1}>1$ we
have $\sum_{n\geq1}\frac{1}{\lambda_{n}^{{(1 - \delta)} (
{2\gamma}%
/({2\gamma-1}) ) }}<\infty$. Hence we have to prove that%
%
%
\begin{equation}
\label{frr1} \int_{0}^{T}\int
_{H} \sum_{n\geq1}\bigl\|
D^{2}u^{(n)}(y)\bigr\|^{4\gamma}\mu_{s}^{x}(dy)\,ds<
\infty.
\end{equation}
Unfortunately, we cannot verify (\ref{frr1}) for an individual
deterministic $x \in H$. On the other hand, by (\ref{cgp}) we know
that, for any $\eta\geq2$,
\[
\int_{H}\bigl\llVert D^{2}u^{(n)} ( z
) \bigr\rrVert^{\eta}\mu( dz ) \leq C_{\eta}\int
_{H}\bigl|B^{(n)}(x)\bigr|^{\eta} \mu(dx),
\]
where $C_{\eta}$ is independent of $n$. Hence we obtain
\begin{eqnarray*}
\int_{H} \sum_{n\geq1}
\bigl\| D^{2}u^{(n)}(y)\bigr\|^{4\gamma}\mu(dy) &\leq&
C_{4\gamma}\int_{H}\sum_{n\geq1}\bigl|B^{(n)}(y)\bigr|^{4\gamma}
\mu(dy)
\\
&\leq& C_{4\gamma} \| B\|_{0}^{4\gamma-2}\int
_{H}\bigl|B(x)\bigr|^{2}%
\mu(dx)\\
&\leq&
C_{4\gamma}\| B\|_{0}^{4\gamma}.
\end{eqnarray*}
%
This estimate is clearly related to (\ref{frr1}) since the law
$\mu_{s}^{x}$ is equivalent to $\mu$ for every $s>0$ and $x$. The
problem is that $\frac{d\mu_{s}^{x}}{d\mu}$ degenerates too
strongly at $s=0$. Therefore we use the fact that
\[
\int_{H} \biggl( \int_{H}f ( z )
\mu_{s}^{x} ( dz ) \biggr) \mu( dx ) =\int
_{H}f ( z ) \mu( dz )
\]
for all $s\geq0$, for every nonnegative measurable function $f$.
Thus, for
any $s\geq0$ with $f ( y ) =\sum_{n\geq1}\|
D^{2}u^{(n)}%
(y)\|^{4\gamma}$, we get
\begin{eqnarray*}
&& \int_{H} \biggl( \int_{H}\sum
_{n\geq1}\bigl\| D^{2}u^{(n)}(y)
\bigr\|^{4\gamma
} \mu_{s}^{x} ( dy ) \biggr) \mu( dx )
\\
&&\qquad =\int_{H}\sum_{n\geq1}\bigl\|
D^{2}u^{(n)}(y)\bigr\|^{4\gamma}\mu( dy ) \\
&&\qquad\leq
C_{4\gamma} \| B\|_{0}^{4\gamma}<\infty.
\end{eqnarray*}

\textit{Step} 4. We prove (\ref{df}) in the case of a random initial
condition $x$ ${\mathcal F}_0$-measurable with law $\mu_{0}$ such
that $\mu_{0}\ll\mu$ and $\int_{H}h_{0}^{\zeta}
\,d\mu<\infty$ for some $\zeta>1$, where $h_{0}:=\frac{d\mu_{0}}%
{d\mu}$.

Denote by $\mu_{s}$ the law of $e^{sA}x+W_{A} ( s )
$, $s\ge0$. We have to prove that
%
%
\begin{equation}
\label{fgt} \int_{0}^{T}\int
_{H} \biggl( \sum_{n\geq1}
\frac{1}{\lambda_{n}^{{1
- \delta}}}\bigl\| D^{2}u^{(n)}(y)\bigr\|^{2}
\biggr)^{2\gamma}\mu_{s}(dy)\,ds <\infty.
\end{equation}
But, since $e^{sA}x$ and $W_A(s)$ are independent, it follows that
\[
\mu_{s}(dz) =\int_{H} \mu_{s}^{y}(dz)
\mu_0 (dy).
\]
Hence, for every Borel measurable $f\dvtx H\rightarrow\mathbb{R}$, if
$\frac{1}{\zeta}+\frac{1}{\zeta^{\prime}}=1$, with $\zeta>1$, we
have%
%
%
\begin{equation}
\label{rit} \int_{H} \bigl|f( y )\bigr| \mu_{s}(dy)\leq
\llVert h_{0}\rrVert_{L^{\zeta} ( \mu) }\llVert f\rrVert_{L^{\zeta
^{\prime}%
} ( \mu) }.
\end{equation}
By (\ref{rit}), we have (similar to step 3)
\begin{eqnarray*}
a_T :\!&=& \int_{0}^{T}\int
_{H} \biggl( \sum_{n\geq1}
\frac{1}{\lambda_{n}^{{1
- \delta}}}\bigl\| D^{2}u^{(n)}(y)\bigr\|^{2}
\biggr)^{2\gamma}\mu_{s}(dy)\,ds
\\
&\leq& T\llVert h_{0}\rrVert_{L^{\zeta} (
\mu) } \biggl(\int
_{H} \biggl( \sum_{n\geq1}
\frac{1}{\lambda_{n}^{{1
- \delta}}}\bigl\| D^{2}u^{(n)}(y)\bigr\|^{2}
\biggr)^{2\gamma\zeta'}\mu(dy)
\biggr)^{1/\zeta'}.
\end{eqnarray*}
By
\begin{eqnarray*}
&& \biggl( \sum_{n\geq1}\frac{1}{\lambda_{n}^{{1 - \delta}}}\bigl\|
D^{2}u^{(n)}%
(y)\bigr\|^{2}
\biggr)^{2\gamma\zeta'}
\\
&&\qquad\leq\biggl( \sum_{n\geq1}\frac{1}{\lambda_{n}^{{(1 - \delta)} (
{2\gamma\zeta'}/({2\gamma\zeta'-1}) ) }}
\biggr)^{2\gamma
\zeta'
-1}\sum_{n\geq1}\bigl\|
D^{2}u^{(n)}(y)\bigr\|^{4\gamma\zeta'} \\
&&\qquad\le C \sum
_{n\geq1}\bigl\| D^{2}u^{(n)}(y)
\bigr\|^{4\gamma\zeta'},
\end{eqnarray*}
we obtain
\[
a_T \leq C T\llVert h_{0}\rrVert_{L^{\zeta} (
\mu) }
\biggl(\int_{H} \sum_{n\geq1}
\bigl\| D^{2}u^{(n)}(y)\bigr\|^{4\gamma\zeta'}
\mu(dy) \biggr)^{1/\zeta'},
\]
which is finite since
\begin{eqnarray*}
&& \int_{H} \sum_{n\geq1} \bigl\|
D^{2}u^{(n)}(y)\bigr\|^{4\gamma
\zeta'}\mu(dy)
\\
&&\qquad \leq C_{4\gamma\zeta^{\prime}} \| B\|_{0}^{4\gamma\zeta^{\prime
}-2}\int
_{H}\bigl|B(x)\bigr|^{2} \mu(dx)\\
&&\qquad\leq C_{4\gamma\zeta^{\prime}}\|
B\|_{0}^{4\gamma\zeta^{\prime}}.
\end{eqnarray*}
The proof is complete.
\end{pf}
%
%
\begin{remark} \label{finite}
Let us comment on the crucial assertion (\ref{cruciale}), that is,
\[
\int_{0}^{T}\int_{H}
\biggl( \sum_{n\geq1}\frac{1}{\lambda_{n}^{{1 -
\delta}}}\bigl\|
D^{2}u^{(n)}(y)\bigr\|^{2} \biggr)^{2\gamma}
\mu_{s}^{x}(dy)\,ds<\infty.
\]
This holds in particular if for some $p >1$
(large enough), we have
%
%
\begin{equation}
\label{gir} \int_0^T R_s |f|
(x) \,ds \le C_{x,T,p} \| f\|_{L^p(\mu)}
\end{equation}
for any $f \in L^p{(\mu)}$ [here $R_t$ is the OU Markov semigroup;
see (\ref{e1})]. \textit{Note that if this assertion holds for any
$x \in H$, then we have pathwise uniqueness for all initial
conditions $x \in H$.} But so far, we could not prove or disprove
(\ref{gir}). We expect, however, that (\ref{gir}) in infinite
dimensions is not true for all $x \in H$.

When $H = \R^d$ one can show that if $p > c(d)$, then (\ref{gir})
holds for any $x \in H$, and so we have uniqueness for all initial
conditions. Therefore, in finite dimension, our method could also
provide an alternative approach to the
Veretennikov result.
In this respect, note that in
finite dimension the SDE $dX_t = b(X_t)\,dt + dW_t$ is equivalent to
$dX_t = - X_t \,dt +
(b(X_t)+ X_t)
\,dt + dW_t$ which is in the form (\ref{SPDE})
with $A = -I$,
but with linearly growing drift term $B(x) = b(x) + x$. Strictly
speaking, we can only recover Veretennikov's result if we realize
the generalization mentioned in Remark~\ref{cf}(i) below.
In this alternative approach,
basically the elliptic $L^p$-estimates with
respect to the Lebesgue measure used in~\cite{Ver} are replaced by
elliptic $L^p{(\mu)}$-estimates using the Girsanov theorem.

Let us check (\ref{gir}) when $H = \R^d$
and $x=0$ for simplicity. By~\cite{DZ1}, Lem\-ma~10.3.3, we know that
\[
R_t f(x) = \int_H f(y) k_t(x,y)
\mu(dy)
\]
and moreover, according to~\cite{DZ1}, Lemma 10.3.8,
for ${p'} \ge1$,
\begin{eqnarray*}
&&\biggl( \int_H \bigl(k_t(0,y)
\bigr)^{p'} \mu(dy) \biggr)^{1/{p'}} \\
&&\qquad= \operatorname{det} \bigl(I -
e^{2tA}\bigr)^{-{1/2} + {1}/({2{p'}})} \operatorname{det} \bigl(I + \bigl({p'}-1
\bigr) e^{2tA}\bigr)^{- {1}/({2{p'}})}.
\end{eqnarray*}
By the H\"older inequality (with $1/{p'} + 1/p =1)$,
\[
\int_0^t R_r f(0) \,dr \le
\biggl( \int_H f(y)^{p} \mu(dy)
\biggr)^{1/p} \int_0^t \biggl( \int
_H \bigl(k_r(0,y)\bigr)^{p'}
\mu(dy) \biggr)^{1/{p'}} \,dr.
\]
Thus (\ref{gir}) holds with $x=0$ if for some ${p'} > 1$ near 1,
%
%
\begin{eqnarray}
\label{ma} && \int_0^t \biggl( \int
_H \bigl(k_r(0,y)\bigr)^{p'}
\mu(dy) \biggr)^{1/{p'}} \,dr
\nonumber\\
&&\qquad= \int_0^t \bigl[\operatorname{det} \bigl(I -
e^{2rA}\bigr)\bigr]^{-{1/2} + {1}/({2{p'}})} \bigl[\operatorname{det} \bigl(I +
\bigl({p'}-1\bigr) e^{2rA}\bigr)\bigr]^{-{1}/({2{p'}})} \,dr
\\
&&\qquad< +\infty.
\nonumber
\end{eqnarray}
It is easy to see that in $\R^d$ there exists $1<K(d) <2$
such that for $1< p' < K(d)$, assertion (\ref{ma}) holds.
\end{remark}
%
%
\begin{remark} \label{cf}
(i) We expect to prove more
generally uniqueness for $B\dvtx H \to H$ which is at most of linear
growth (in particular, bounded on each balls) by using a stopping
time argument.

(ii) We also expect to implement
the uniqueness result
to drifts $B$ which are also time dependent.\vadjust{\goodbreak}
However, to extend our method we need parabolic $L^p_{\mu}$-estimate
involving the Ornstein--Uhlenbeck operator which are not yet available
in the literature.
\end{remark}
%

\section{Examples}\label{sectionexamples}\label{sec4}

We discuss some examples in several steps. First we show a simple
one-dimensional example of wild nonuniqueness due to noncontinuity of the
drift. Then we show two infinite dimensional, very natural
generalizations of
this example. However, both of them do not fit perfectly with our
purposes, so
they are presented mainly to discuss possible phenomena. Finally, in
Section~\ref{subsectionexample4}, we modify the previous examples
in such
a way to get a very large family of \textit{deterministic problems with
nonuniqueness for all initial conditions}, which fits the assumptions
of our
result of uniqueness by noise.\looseness=-1

\subsection{An example in dimension one}
\label{subsectionexample1}\label{sec4.1}

In dimension 1, one of the simplest and more dramatic examples of
nonuniqueness is the equation%
\[
\frac{d}{dt}X_{t}=b_{\mathrm{Dir}} ( X_{t} ),\qquad
X_{0}=x,
\]
where%
\[
b_{\mathrm{Dir}} ( x ) = \cases{ %
1, &\quad if $x\in\mathbb{R}\setminus
\mathbb{Q}$,
\cr
0, &\quad if $x\in\mathbb{Q}$}
\]
(the so-called Dirichlet function). Let us call solution any continuous
function $X_{t}$ such that
\[
X_{t}=x+\int_{0}^{t}b_{\mathrm{Dir}}
( X_{s} ) \,ds
\]
for all $t\geq0$. For every $x$, the function
\[
X_{t}=x+t
\]
is a solution: indeed, $X_{s}\in\mathbb{R}\setminus\mathbb{Q}$ for
a.e. $s$,
hence $b_{\mathrm{Dir}} ( X_{s} ) =1$ for a.e. $s$, hence $\int_{0}%
^{t}b_{\mathrm{Dir}} ( X_{s} ) \,ds=t$ for all $t\geq0$. But from
$x\in\mathbb{Q}$ we have also the solution
\[
\widetilde{X}_{t}=x,
\]
because $\widetilde{X}_{s}\in\mathbb{Q}$ for all $s\geq0$ and thus
$b_{\mathrm{Dir}} ( X_{s} ) =0$ for all $s\geq0$. Therefore, we have
nonuniqueness from every initial condition $x\in\mathbb{Q}$. Not
only: for every
$x$ and every $\varepsilon>0$, there are infinitely many solutions on
$ [
0,\varepsilon] $. Indeed, one can start with the solution $X_{t}=x+t$
and branch at any $t_{0}\in[ 0,\varepsilon] $ such that
$x+t_{0}\in\mathbb{Q}$, continuing with the constant solution.
Therefore, in a
sense, there is nonuniqueness from every initial condition.

\subsection{First infinite-dimensional generalization (not of
parabolic type)}\label{sec4.2}

This example can be immediately generalized to infinite dimensions by taking
$H=l^{2}$ (the space of square summable sequences),
\[
B_{\mathrm{Dir}} ( x ) =\sum_{n=1}^{\infty}
\alpha_{n}b_{\mathrm{Dir}} ( x_{n} ) e_{n},\vadjust{\goodbreak}
\]
where $x= ( x_{n} ) $, $ ( e_{n} ) $ is the canonical
basis of $H$, and $\alpha_{n}$ are positive real numbers such that
$\sum_{n=1}^{\infty}\alpha_{n}^{2}<\infty$. The mapping $B$ is well
defined from
$H$ to $H$, it is Borel measurable and bounded, but of course not continuous.
Given an initial condition $x= ( x_{n} ) \in H$, if a function
$X ( t ) = ( X_{n} ( t ) ) $ has all
components $X_{n} ( t ) $ which satisfy%
\[
X_{n} ( t ) =x_{n}+\int_{0}^{t}
\alpha_{n}b_{\mathrm{Dir}} \bigl( X_{n} ( s ) \bigr) \,ds,
\]
then $X ( t ) \in H$ and is continuous in $H$ (we see this
from the
previous identity), and satisfies
\[
X ( t ) =\sum_{n=1}^{\infty}X_{n}
( t ) e_{n}=x+\int_{0}^{t}B_{\mathrm{Dir}}
\bigl( X ( s ) \bigr) \,ds.
\]
So we see that this equation has infinitely many solutions from every initial
condition.

Unfortunately our theory of regularization by noise cannot treat this simple
example of nonuniqueness, because we need a regularizing operator $A$
in the
equation to compensate for the regularity troubles introduced by a
cylindrical noise.

\subsection{Second infinite-dimensional generalization
(nonuniqueness only for
a few initial conditions)}\label{sec4.3}

Let us start in the most obvious way. Namely, consider the equation in
$H=l^{2}$
\[
X ( t ) =e^{tA}x+\int_{0}^{t}e^{ ( t-s ) A}%
B_{\mathrm{Dir}} \bigl( X ( s ) \bigr) \,ds,
\]
where
\[
Ax=-\sum_{n=1}^{\infty}\lambda_{n}
\langle x,e_{n} \rangle e_{n}%
\]
with $\lambda_{n}>0$, $\sum_{n=1}^{\infty}\frac{1}{\lambda
_{n}^{1-\varepsilon
_{0}}}<\infty$. Componentwise we have%
\[
X_{n} ( t ) =e^{-t\lambda_{n}}x_{n}+\int
_{0}^{t}e^{- (
t-s ) \lambda_{n}}b_{\mathrm{Dir}} \bigl(
X_{n} ( s ) \bigr) \,ds.
\]
For $x= ( x_{n} ) \in H$ with \textit{all} nonzero components
$x_{n}$, the solution is unique, with components
\[
X_{n} ( t ) =e^{-t\lambda_{n}}x_{n}+
\frac{1-e^{-t\lambda
_{n}}%
}{\lambda_{n}}%
\]
[we have $X_{n} ( s ) \in\mathbb{R}\setminus\mathbb{Q}$
for a.e.
$s$, hence $b_{\mathrm{Dir}} ( X_{n} ( s ) ) =1$; and it is
impossible to keep a solution constant on a rational value, due to the term
$e^{-t\lambda_{n}}x_{n}$ which always appears]. This is also a
solution for all
$x$.

But from any initial condition $x= ( x_{n} ) \in H$ such
that at
least one component $x_{n_{0}}$ is zero,\vadjust{\goodbreak} we have at least two
solutions: the
previous one and any solution such that
\[
X_{n_{0}} ( t ) =0.
\]

This example fits our theory in the sense that all assumptions are satisfied,
so our main theorem of ``uniqueness by
noise''
applies. However, our theorem states only that uniqueness is restored for
$\mu$-a.e. $x$, where $\mu$ is the invariant Gaussian measure of the linear
stochastic problem, supported on the whole $H$. We already know that this
deterministic problem has uniqueness for $\mu$-a.e. $x$: it has a unique
solution for all $x$ with all components different from zero. Therefore our
theorem is not empty but not competitive with the deterministic theory, for
this example.

\subsection{Infinite-dimensional examples with wild
nonuniqueness}\label{subsectionexample4}\label{sec4.4}

Let $H$ be a separable Hilbert space with a complete orthonormal system
$ ( e_{n} ) $. Let $A$ be as in the assumptions of this paper.
Assume\vspace*{1pt} that $e_{1}$ is eigenvector of $A$ with eigenvalue $-\lambda
_{1}$. Let
$\widetilde{H}$ be the orthogonal to $e_{1}$ in $H$, the span of
$e_{2},e_{3},\ldots,$ and let $\widetilde{B}\dvtx\widetilde
{H}\rightarrow
\widetilde{H}$ be
Borel measurable and bounded. Consider $\widetilde{B}$ as an operator
in $H$,
by setting $\widetilde{B} ( x ) = \widetilde{B} (
\sum_{n=2}^{\infty}x_{n}e_{n} ) $.

Let $B$ be defined as%
\[
B ( x ) = \bigl( ( \lambda_{1}x_{1} ) \wedge
1+b_{\mathrm{Dir}} ( x_{1} ) \bigr) e_{1}+\widetilde{B} (
x )
\]
for all $x= ( x_{n} ) \in H$. Then $B\dvtx H\rightarrow H$ is Borel
measurable and bounded. The deterministic equation for the first component
$X_{1} ( t ) $ is, in differential form,%
\begin{eqnarray*}
\frac{d}{dt}X_{1} ( t ) & = & -\lambda_{1}X_{1}
( t ) +\lambda_{1}X_{1} ( t ) +b_{\mathrm{Dir}} \bigl(
X_{1} ( t ) \bigr)
\\
& = &b_{\mathrm{Dir}} \bigl( X_{1} ( t ) \bigr)
\end{eqnarray*}
as soon as
\[
\bigl\llvert X_{1} ( t ) \bigr\rrvert\leq1/\lambda_{1}.
\]
In other words, the full drift $Ax+B ( x ) $ is given, on
$\widetilde{H}$, by a completely general scheme coherent with our assumption
(which may have deterministic uniqueness or not); and along $e_{1}$ it
is the
Dirichlet example of Section~\ref{subsectionexample1}, at least
as soon
as a solution satisfies $\llvert X_{1} ( t ) \rrvert
\leq1/\lambda_{1}$.

Start from an initial condition $x$ such that
\[
\llvert x_{1}\rrvert<1/\lambda_{1}.
\]
Then, by continuity of trajectories and the fact that any possible
solution to
the equation satisfies the inequality
\[
\biggl\llvert\frac{d}{dt}X_{1} ( t ) \biggr\rrvert\leq
\lambda_{1}\bigl\llvert X_{1} ( t ) \bigr\rrvert+2,
\]
there exists $\tau>0$ such that for every possible solution, we have
\[
\bigl\llvert X_{1} ( t ) \bigr\rrvert<1/\lambda_{1}
\qquad\mbox{for all }t\in[ 0,\tau].\vadjust{\goodbreak}
\]
So, on $ [ 0,\tau] $, all solutions solve $\frac{d}{dt}%
X_{1} ( t ) =b_{\mathrm{Dir}} ( X_{1} ( t ) ) $ which has infinitely many
solutions (step 1). Therefore also the infinite-dimensional equation
has infinitely many solutions.

We have proved that nonuniqueness holds for all $x\in H$ such that $x_{1}$
satisfies $\llvert x_{1}\rrvert<1/\lambda_{1}$. This set of initial
conditions has positive $\mu$-measure; hence we have a class of
examples of
deterministic equations where nonuniqueness holds for a set of initial
conditions with positive $\mu$-measure. Our theorem applies and states for
$\mu$-a.e. such initial condition we have uniqueness by noise.

\begin{appendix}\label{app}

\section*{Appendix}

\subsection{Girsanov's theorem in infinite dimensions with respect
to a cylindrical Wiener process}\label{aapA.1}


In the main body of the paper, the Girsanov theorem for SDEs on
Hilbert spaces of type (\ref{SPDE}) with cylindrical Wiener noise is
absolutely crucial. Since a complete and reasonably self-contained
proof is hard to find in the literature, for the convenience of the
reader, we give a detailed proof of this folklore result (see,
e.g.,~\cite{DZ,Ferr10,GG} and~\cite{Fe}) in our situation,
but \textit{even for at most linearly growing $B$.}
The proof is reduced
to the Girsanov theorem of general real valued continuous local
martingales; see~\cite{RY}, (1.7) Theorem, page 329.

We consider the situation of the main body of the paper, that is, we
are given a negative definite self-adjoint operator $A\dvtx D(A)
\subset H \to H$ on a separable Hilbert space $(H, \langle\cdot,
\cdot\rangle)$ with $(-A)^{-1 + \delta}$ being of trace class, for
some $\delta\in(0,1)$, a measurable map $B\dvtx H \to H$ of at most
linear growth and $W$ a cylindrical Wiener process over $H$ defined
on a filtered probability space
$(\Omega, {\mathcal F}, {\mathcal F}_t, \P)$ represented in terms of the
eigenbasis $\{ e_k\}_{k \in\mathbb{N}}$ of $(A, D(A))$ through
a sequence
%
%
\begin{equation}
\label{A0} W(t) = \bigl(\beta_k (t) e_k
\bigr)_{k\in\N},\qquad t\ge0,
\end{equation}
where $\beta_k$, $k \in\N$, are independent real valued Brownian
motions starting at zero on $(\Omega, {\mathcal F}, {\mathcal F}_t,
\P)$. Consider the stochastic equations
%
%
\begin{equation}
\label{A1} dX(t) = \bigl(AX(t) + B\bigl(X(t)\bigr)\bigr)\,dt + dW(t),\qquad t
\in[0,T],
X(0)=x,
\end{equation}
and
%
%
\begin{equation}
\label{A2} dZ(t) = AZ(t)\,dt + dW(t),\qquad t \in[0,T], Z(0)=x,
\end{equation}
for some $T>0$.
%
%
\begin{theorem} \label{tA1}
Let $x \in H$. Then (\ref{A1}) has a unique weak mild solution, and
its law $\P_x$ on $C([0,T]; H)$ is equivalent to the law $ \Q_x$ of
the solution to (\ref{A2}) (which is just the classical OU process).
If $B$ is bounded, $x$ may be replaced by an ${\mathcal
F}_0$-measurable $H$-valued random variable.\vadjust{\goodbreak}
\end{theorem}

The rest of this section is devoted to the proof of this theorem.
We first need some preparation and start with recalling that
because Tr$[(-A)^{-1 + \delta}] < \infty$, $\delta\in(0,1)$, the
stochastic convolution
%
%
\begin{equation}
\label{A3} W_A(t):= \int_0^t
e^{(t-s)A} \,dW(s),\qquad t \ge0,
\end{equation}
is a well defined ${\mathcal F}_t$-adapted stochastic process (``OU
process'') with continuous paths in $H$ and
%
%
\begin{equation}
\label{A4} Z(t,x):= e^{tA} x + W_A(t),\qquad t \in[0,T],
\end{equation}
is the unique mild solution of (\ref{A1}).

Let $b(t)$, $t \ge0$, be a progressively measurable
$H$-valued process on $(\Omega, {\mathcal F},\break {\mathcal F}_t, \P)$
such that
%
%
\begin{equation}
\label{A5} {\mathbb E} \biggl[\int_0^T
\bigl|b(s)\bigr|^2 \,ds \biggr] < \infty
\end{equation}
and
%
%
\begin{equation}
\label{A6} X(t,x):= Z(t,x) + \int_0^t
e^{(t-s)A} b(s) \,ds,\qquad t \in[0,T].
\end{equation}
We set
%
%
\begin{equation}
\label{A7} W_k(t):= \beta_k(t) e_k,\qquad t
\in[0,T], k \in\N,
\end{equation}
and define
%
%
\begin{equation}
\label{A8} Y(t):= \int_0^t \bigl\langle b(s),
dW(s) \bigr\rangle:= \sum_{k \ge1} \int
_0^t \bigl\langle b(s), e_k \bigr
\rangle\,dW_k(s),\qquad t \in[0,T].
\end{equation}
%
%
\begin{lemma} \label{lA2} The series on the right-hand side
of
(\ref{A8}) converges in $L^2(\Omega,\break \P; C([0,T]; \R))$. Hence the
stochastic integral $Y(t)$ is well defined and a continuous
real-valued martingale, which is square integrable.
\end{lemma}
\begin{pf}
We have for all $n, m \in\N$, $n >m$, by Doob's inequality,
\begin{eqnarray*}
&& {\mathbb E} \Biggl[ \sup_{t \in[0,T]} \Biggl| \sum_{k=m}^n
\int_0^t \bigl\langle e_k, b(s)
\bigr\rangle\,dW_k(s) \Biggr|^2 \Biggr]\\
&&\qquad \le2{\mathbb E} \Biggl[
\Biggl| \sum_{k=m}^n \int_0^T
\bigl\langle e_k, b(s)\bigr\rangle\,dW_k(s)
\Biggr|^2 \Biggr]
\\
&&\qquad= 2\sum_{k,l =m}^n {\mathbb E}
\biggl[ \int_0^T \bigl\langle
e_k, b(s)\bigr\rangle\,dW_k(s) \int_0^T
\bigl\langle e_l, b(s)\bigr\rangle\,dW_l(s) \biggr]
\\
&&\qquad= 2\sum_{k =m}^n {\mathbb E}
\biggl[ \int_0^T \bigl\langle
e_k, b(s)\bigr\rangle^2 \,ds \biggr] \to0
\end{eqnarray*}
as $m,n \to\infty$ because of (\ref{A5}). Hence the series on the
right-hand side of (\ref{A8}) converges in $L^2(\Omega, \P; C([0,T];
\R))$, and the assertion follows.
\end{pf}
%
%
\begin{remark}
\label{rA3} It can be shown that if $ \int_0^t\langle b(s),dW(s)
\rangle, t\in[0,T]$, is defined as usual, using approximations by
elementary functions (see~\cite{PrevRoeckner}, Lem\-ma~2.4.2), the
resulting process is the same.
\end{remark}

It is now easy to calculate the corresponding variation process
$\langle\int_0^\cdot\langle b(s),\break dW(s) \rangle\rangle_t$, $t\in[0,T]$.
%
%
\begin{lemma}
\label{lA4}
We have
\[
\langle Y\rangle_t=\biggl\langle\int_0^\cdot
\bigl\langle b(s),dW(s) \bigr\rangle\biggr\rangle_t=\int
_0^t\bigl|b(s)\bigr|^2\,ds,\qquad t\in[0,T].
\]
\end{lemma}
\begin{pf}
We have to show that
\[
Y^2(t)-\int_0^t\bigl|b(s)\bigr|^2\,ds,\qquad
t\in[0,T],
\]
is a martingale, that is, for all bounded $\mathcal F_t$-stopping times
$\tau$,
we have
\[
\mathbb{\mathbb E}\bigl[Y^2(\tau)\bigr]= \mathbb{\mathbb E} \biggl[
\int_0^\tau\bigl|b(s)\bigr|^2\,ds \biggr],
\]
which follows immediately as in the proof of Lemma~\ref{lA2}.
\end{pf}

Define the measure
%
%
\begin{equation}
\label{eA9} \widetilde\P:=e^{Y(T)-\langle Y\rangle_t/2}\cdot\P
\end{equation}
on $(\Omega,{\mathcal F})$, which is equivalent to $\P$. Since
$\mathcal E (t):=e^{Y(T)-\langle Y\rangle_t/2}, t\in[0,T]$, is a
nonnegative local martingale, it follows by Fatou's lemma that it is
a supermartingale, and since $\mathcal E (0)=1$, we have
\[
\mathbb E\bigl[\mathcal E (t)\bigr]\le\mathbb E\bigl[\mathcal E (0)\bigr]=1.
\]
Hence $\widetilde\P$ is a sub-probability measure.
%
%
\begin{proposition}
\label{pA5}
Suppose that $\widetilde\P$ is a probability measure, that is,
%
%
\begin{equation}
\label{eA10} \mathbb E\bigl[\mathcal E (T)\bigr]=1.
\end{equation}
Then
\[
\widetilde W_k(t):=W_k(t)-\int_0^t
\bigl\langle e_k,b(s) \bigr\rangle\,ds,\qquad t\in[0,T], k\in\N,
\]
are independent real-valued Brownian motions starting at $0$ on
$(\Omega, {\mathcal F}, (\mathcal F_t),\widetilde{\mathbb P})$, that is,
\[
\widetilde W(t):=\bigl(\widetilde W(t)e_k\bigr)_{k\in\N},\qquad t
\in[0,T],
\]
is a cylindrical Wiener process over $H$ on $(\Omega,
{\mathcal F}, (\mathcal
F_t),\widetilde{\mathbb P})$.
\end{proposition}
\begin{pf}
By the classical Girsanov theorem (for general real-valued martingales,
see~\cite{RY}, (1.7) Theorem, page 329), it follows that for every
$k\in\N$,
\[
W_k(t)-\langle W_k,Y \rangle_t,\qquad t\in[0,T],
\]
is a local martingale under $\widetilde{\mathbb P}$. Set
\[
Y_n(t):=\sum_{k=1}^n\int
_0^t\bigl\langle e_k,b(s) \bigr
\rangle\,dW_k(s),\qquad t\in[0,T], n\in\N.
\]
Then by Cauchy--Schwartz's inequality
\[
\bigl|\langle W_k,Y-Y_n \rangle\bigr|_t=\langle
W_k \rangle_t^{1/2}\langle Y-Y_n
\rangle_t^{1/2},\qquad t\in[0,T],
\]
and since
\[
\mathbb E\bigl[\langle Y-Y_n \rangle_t\bigr]=\mathbb E
\bigl[(Y-Y_n)^2\bigr]\to0 \qquad\mbox{as } n\to\infty
\]
by Lemma~\ref{lA2}, we conclude that (selecting a subsequence if
necessary) $\P$-a.s. for all $t\in[0,T]$
\[
\langle W_k,Y \rangle_t=\lim_{n\to\infty}\langle
W_k,Y_n \rangle_t =\int_0^t
\bigl\langle e_k,b(s) \bigr\rangle\,ds,
\]
since\vspace*{1pt} $\langle W_k,W_l \rangle_t=0$ if $k\neq l$, by independence.
Hence each $\widetilde{W}_k$ is a local martingale under $\widetilde
{\P}$.

It remains to show that for every $n\in\N$, $(\widetilde
W_1,\ldots,\widetilde W_n)$ is, under $\widetilde{\P}$, an
$n$-dimensional Brownian motion. But
$\P$-a.s. for $l\neq k$
\[
\langle\widetilde W_l,\widetilde W_k
\rangle_t=\langle W_l,W_k
\rangle_t=\delta_{l,k}(t),\qquad t\in[0,T].
\]
Since $\P$ is equivalent to $\widetilde{\P}$, this also holds
$\widetilde{\P}$-a.s. Hence\vspace*{1pt} by L\`evy's characterization theorem
(\cite{RY}, (3.6) Theorem, page 150) it follows that
$(\widetilde W_1,\ldots,\widetilde W_n)$ is an $n$-dimensional Brownian
motion in $\R^n$ for all $n$, under $\widetilde{\P}$.
\end{pf}
%
%
\begin{proposition}
\label{pA6}
Let $W_A(t), t\in[0,T]$,
be defined as in (\ref{A3}). Then there exists $\epsilon>0$ such that
\[
\E\Bigl[\exp\Bigl\{\epsilon\sup_{t\in[0,T]} \bigl|W_A(t)\bigr| \Bigr
\}^2 \Bigr]<\infty.
\]
\end{proposition}
\begin{pf}
Consider the distribution $\Q_0:=\P\circ W_A^{-1}$ of $W_A$ on
$E:=C([0,T];\break H)$. If $\Q_0$ is a Gaussian measure on $E$, the
assertion follows by Fernique's theorem. To show that $\Q_0$ is a
Gaussian measure on $E$, we have to show that for each $l$ in the
dual space $E'$ of $E$, we have that $\Q_0 \circ l^{-1}$ is Gaussian
on $\R$. We prove this in two steps.\vspace*{8pt}

\textit{Step} 1. Let $t_0\in[0,T]$, $h\in H$ and $\ell(\omega
):=\langle h,\omega(t_0)\rangle$ for $\omega\in E$. To see that $\Q_0
\circ\ell^{-1}$ is Gaussian on $\R$, consider a sequence $\delta_k\in
C([0,T];\R), k\in\N$,
such that $\delta_k(t)\,dt$ converges weakly to the Dirac measure
$\epsilon_{t_0}$.
Then for all
$\omega\in E$,
\[
\ell(\omega)=\lim_{k\to\infty}\int_0^T
\bigl\langle h,\omega(s)\bigr\rangle\delta_k(s)\,ds =
\lim_{k\to\infty} \int_0^T\bigl\langle h
\delta_k(s),\omega(s) \bigr\rangle\,ds.
\]
Since (e.g., by~\cite{D}, Proposition 2.15, the law of $W_A$ in
$L^2(0,T;H)$ is Gaussian, it follows that the distribution of $\ell$
is Gaussian.\vspace*{8pt}

\textit{Step} 2. The following argument is taken from
\cite{DL}, Proposition A.2. Let $\omega\in E$; then we can consider its
Bernstein approximation
\[
\beta_n(\omega) (t):=\sum_{k=1}^n
\pmatrix{n
\cr
k}\varphi_{k,n}(t)\omega(k/n),\qquad n\in\N,
\]
where $\varphi_{k,n}(t):=t^k(1-t)^{n-k}$. But the linear map
\[
H\ni x \to\ell(x\varphi_{k,n})\in\R
\]
is continuous in $H$, and hence there exists $h_{k,n}\in H$ such that
\[
\ell(x\varphi_{k,n})=\langle h_{k,n},x \rangle,\qquad x\in H.
\]
Since $\beta_n(\omega)\to\omega$ uniformly for all $\omega\in E$,
it follows that for all $\omega\in E$,
\[
\ell(\omega)=\lim_{n\to\infty}\ell\bigl(\beta_n(\omega)\bigr)=
\lim_{n\to
\infty}\sum_{k=1}^n
\pmatrix{n
\cr
k}\bigl\langle h_{k,n},\omega(k/n) \bigr\rangle,\qquad n\in\N.
\]
Hence it follows by step 1 that $\Q_0\circ l^{-1}$ is
Gaussian.
\end{pf}

Now we turn to SDE (\ref{A1}) and define
%
%
\begin{eqnarray}
\label{eA11} %
M&:=&e^{\int_0^T \langle B(e^{tA}x+W_A(t)),dW(t) \rangle-
(1/2)\int_0^T|B(e^{tA}x+W_A(t))|^2\,dt},
\nonumber\\[-8pt]\\[-8pt]
\widetilde{\P}_x&:=&M\P.
\nonumber
\end{eqnarray}
Obviously, Proposition~\ref{pA7} below implies (\ref{A5}) and that
hence $M$ is well defined.
%
%
\begin{proposition}
\label{pA7} $\widetilde{\P}_x$ is a probability measure on
$(\Omega,{\mathcal F})$, that is,\break \mbox{$\E(M)=1$}.
\end{proposition}
\begin{pf}
As before we set $Z(t,x):=e^{tA}x+W_A(t), t\in[0,T]$. By
Proposition~\ref{pA6} the arguments below are standard (see, e.g.,
\cite{KS}, Corollaries 5.14 and 5.16, pages 199 and 200). Since $B$ is
of at most linear growth, by\vspace*{2pt} Proposition~\ref{pA6} we can find $N\in\N$
large enough such that for all $0\le i\le N$ and $t_i:=\frac{iT}N$,
\[
\E\bigl[ e^{(1/2)\int_{t_{i-1}}^{t_i}
|B(e^{tA}x+W_A(t))|^2\,dt} \bigr]<\infty.
\]
Defining $B_i(e^{tA}x+W_A(t)):=\one
_{(t_{i-1},t_i]}(t)B(e^{tA}x+W_A(t))$, it follows from Novikov's
criterion (\cite{RY}, (1.16) Corollary, page 333)
that for all $1\le i\le N$,
\[
\mathcal E_i (t):=e^{(1/2)\int_{0}^{t} \langle
B_i(e^{sA}x+W_A(s)),dW(s) \rangle-(1/2) \int
_0^t|B_i(e^{sA}x+W_A(s)|^2\,ds },\qquad t\in[0,T],
\]
is an $\mathcal F_t$-martingale under $\P$. But then since $\mathcal
E_i (t_{i-1})=1$, by the martingale property of each $\mathcal E_i $,
we can conclude that
\begin{eqnarray*}
&&\E\bigl[ e^{ \int_{0}^{t} \langle B(e^{sA}x+W_A(s)),dW(s)
\rangle-(1/2) \int_0^t|B(e^{sA}x+W_A(s)|^2\,ds } \bigr]
\\
&&\qquad=\E\bigl[\mathcal E_N (t_N)\mathcal
E_{N-1} (t_{N-1})\cdots\mathcal E_1
(t_1) \bigr]
\\
&&\qquad=\E\bigl[\mathcal E_N (t_{N-1})\mathcal
E_{N-1} (t_{N-1})\cdots\mathcal E_1
(t_1) \bigr]
\\
&&\qquad=\E\bigl[\mathcal E_{N-1} (t_{N-1})\cdots\mathcal
E_1 (t_1) \bigr]
\\
&&\qquad\hspace*{3.4pt}\vdots
\\
&&\qquad=\E\bigl[ \mathcal E_1 (t_1) \bigr]=\E\bigl[\mathcal
E_1 (t_0) \bigr] =1.
\end{eqnarray*}
\upqed
\end{pf}
%
%
\begin{remark} \label{obvio}
It is obvious from the previous proof that $x$
may always be replaced by an ${\mathcal F}_0$-measurable $H$-valued
r.v. which is exponentially integrable, and by any ${\mathcal
F}_0$-measurable $H$-valued r.v. if $B$ is bounded. The same holds
for the rest of the proof of Theorem~\ref{tA1}, that is, the following
two propositions.
\end{remark}
%
%
\begin{proposition}
\label{pA9}
We have $\widetilde{\P}_x$-a.s.
%
%
\begin{eqnarray}
\label{eA12} Z(t,x)&=&e^{tA}x+\int_0^te^{(t-s)A}B
\bigl(Z(s,x)\bigr)\,ds\nonumber\\[-8pt]\\[-8pt]
&&{}+\int_0^te^{(t-s)A}\,d
\widetilde{W}(s),\qquad t\in[0,T],\nonumber
\end{eqnarray}
where $\widetilde{W}$ is the cylindrical Wiener process under
$\widetilde{\P}_x$ introduced in Proposition~\ref{pA5} with
$b(s):=B(Z(s,x))$, which applies because of Proposition~\ref{pA7},
that is, under~$\widetilde{\P}_x$, $Z(\cdot,x)$ is a mild solution of
\[
dZ(t)=\bigl(AZ(t)+B\bigl(Z(t)\bigr)\bigr)\,dt+d\widetilde{W}(t),\qquad t\in[0,T],
Z(0)=x.
\]
\end{proposition}
\begin{pf}
Since $B$ is of at most linear growth and because of Proposition~\ref{pA6},
to prove (\ref{eA12}), it is enough to show that for all $k\in\N$
and $x_k:=\langle e_k,x \rangle$ for $x\in H$ we have, since
$Ae_k=-\lambda_k e_k$, that
\[
dZ_k(t,x)=\bigl(-\lambda_k Z_k(t,x)+B_k
\bigl(Z(t,x)\bigr)\bigr)\,dt+d\widetilde{W}_k(t),\qquad t\in[0,T], Z(0)=x.
\]
But this is obvious by the definition of $\widetilde{W}_k$.
\end{pf}

Proposition~\ref{pA9} settles the existence part of Theorem \ref
{tA1}. Now let us turn to the uniqueness part and complete the proof of
Theorem~\ref{tA1}.
%
%
\begin{proposition}
\label{pA10}
The weak solution to (\ref{A1}) constructed above is unique and its
law is equivalent to $\Q_x$
with density in $L^p(\Omega, \P)$ for all $p \ge1$.
\end{proposition}
\begin{pf}
Let $X(t,x), t\in[0,T]$, be a weak solution to (\ref{A1}) on a
filtered probability space $(\Omega, {\mathcal F}, (\mathcal
F_t),\P)$ for a cylindrical process of type (\ref{A0}). Hence
\[
X(t,x)=e^{tA}x+W_A(t)+\int_0^te^{(t-s)A}B
\bigl(X(s,x)\bigr)\,ds.
\]
Since $B$ is at most of linear growth, it follows from Gronwall's
inequality that for some
constant $C\ge0$,
\[
\sup_{t\in[0,T]}\bigl|X(t,x)\bigr|\le C_1 \Bigl(1+\sup_{t\in
[0,T]}\bigl|e^{tA}x+W_A(t)\bigr|
\Bigr).
\]
Hence by Proposition~\ref{pA6},
%
%
\begin{equation}
\label{74} \E\Bigl[\exp\Bigl\{\epsilon\sup_{t\in[0,T]}\bigl|X(t,x)\bigr| \Bigr
\}^2 \Bigr]<\infty.
\end{equation}
Define
\[
M:=e^{-\int_0^T \langle B(X(s,x)),dW(s) \rangle-(1/2)\int
_0^T|B(X(s,x))|^2\,ds}
\]
and $\widetilde{\P}:=M\cdot\P$. Then by exactly the same
arguments as above,
\[
\E[M]=1.
\]
Hence by Proposition~\ref{pA5}, defining
\[
\widetilde{W}_k(t):=W_k(t)+\int_0^t
\bigl\langle e_k,B\bigl(X(s,x)\bigr) \bigr\rangle\,ds,\qquad
t\in[0,T], k\in
\N,
\]
we obtain that $\widetilde{W}(t):=(\widetilde{W}_k(t)e_k)_{k\in\N}$
is a cylindrical Wiener process under $\widetilde{\P}$ and thus
\[
\widetilde{W}_A(t):=\int_0^te^{(t-s)A}\,d
\widetilde{W}(s)=W_A(t)+ \int_0^te^{(t-s)A}B
\bigl(X(s,x)\bigr)\,ds,\qquad t\in[0,T],
\]
and therefore,
\[
X(t,x)=e^{tA}x+\widetilde{W}_A(t),\qquad t\in[0,T],
\]
is an Ornstein--Uhlenbeck process under $\widetilde{\P}$ starting at
$x$. But since it is easy to see that
\[
\int_0^T\bigl\langle B\bigl(X(s,x)
\bigr),dW(s)\bigr\rangle=\int_0^T\bigl\langle
B\bigl(X(s,x)\bigr),d\widetilde{W}(s)\bigr\rangle-\int_0^T\bigl|B
\bigl(X(s,x)\bigr)\bigr|^2\,ds,
\]
it follows that
\[
\P=e^{\int_0^T\langle B(X(s,x)),\widetilde{W}(s)\rangle-
(1/2)\int_0^T|B(X(s,x))|^2\,ds}\cdot\widetilde{\P}.
\]
Since
\[
\widetilde{W}_k(t)=\bigl\langle e_k,
\widetilde{W}_A(t) \bigr\rangle+\lambda_k\int
_0^t\bigl\langle e_k,
\widetilde{W}_A(s)\bigr\rangle\,ds,
\]
and since $ X(s,x)=e^{sA}x+\widetilde{W}_A(s), $ it follows that
$\int_0^T\langle B(X(s,x)),d\widetilde{W}(s)\rangle$ is measurable
with respect to the $\sigma$-algebra generated by $\widetilde{W}_A$.
Hence $ \frac{d\P}{d\widetilde{\P}}=\rho_x(X(\cdot,x))$ for some
$\rho_x\in\mathcal B (C([0,T];H) )$, and thus setting
$\Q_x:=\P_x\circ X(\cdot,x)^{-1}$, we get
\[
\P_x:=\P\circ X(\cdot,x)^{-1}=\rho_x
\Q_x.
\]

But since it is well known that the mild solution of (\ref{A2}) is
unique in distribution,
the assertion follows, because clearly $\rho_x>0$, $\Q_ x$-a.s.
\end{pf}

\subsection{A useful lemma}\label{aapA.2}
%
%
\begin{lemma} \label{da}
Let $f \in W^{1,2}(H, \mu) \cap C_b(H)$.
Let $X = (X_t)$ and $Y = (Y_t)$ be
two solutions to (\ref{SPDE}) starting from a deterministic
$x \in H$ or from
a r.v. $x $ as in Theorem~\ref{maintheorem}. Let $t \ge0$.
Then for $dt \otimes\P$-a.e. $(t, \omega)$, we have
%
%
\begin{equation}
\label{75}
\int_0^1 \bigl| Df \bigl(r
X_t(\omega) + (1-r)Y_t(\omega)\bigr)\bigr| \,dr < \infty
\end{equation}
and
\begin{eqnarray}
\label{tes2}
&&
f\bigl(X_t(\omega)\bigr) - f
\bigl(Y_t(\omega)\bigr)
\nonumber\\[-8pt]\\[-8pt]
&&\qquad = \int_0^1 \bigl\langle Df \bigl(r
X_t(\omega) + (1-r)Y_t(\omega)\bigr), X_t(
\omega)- Y_t(\omega) \bigr\rangle\,dr.
\nonumber
\end{eqnarray}
\end{lemma}
\begin{pf} Formula (\ref{tes2}) is meaningful if we consider a
Borel representative of $Df \in L^2(\mu)$; that is, we consider a
Borel function $g\dvtx H \to H$ such that $g = Df$, $\mu$-a.e.

Clearly the right-hand side of (\ref{tes2}) is independent of this chosen
version because (setting again $Z_t^r = r X_t + (1-r)Y_t$) it is
equal to
\[
\biggl\langle\int_0^1 Df
\bigl(Z_t^r (\omega)\bigr) \,dr, X_t(\omega)-
Y_t (\omega) \biggr\rangle,
\]
and for a Borel function $g\dvtx H \to H$ with
$g=0$ $\mu$-a.e., we have, for any $T>0$, $\epsilon\in(0,T]$,
\[
{\mathbb E} \biggl[ \int_{\epsilon}^T\int
_0^1 \bigl| g \bigl(Z_t^r
\bigr)\bigr| \,dr \,dt \biggr] = \int_{\epsilon}^T \int
_0^1 {\mathbb E} \bigl|g \bigl(Z_t^r
\bigr)\bigr| \,dr \,dt =0,
\]
since by the Girsanov theorem (see Theorem~\ref{tA1})
the law of the
r.v. $Z_t^r$, $r \in[0,1]$, is absolutely continuous with respect
to $\mu$.

Let us prove (\ref{tes2}). By~\cite{DZ1}, Section 9.2, there exists a
sequence of functions $(f_n) \subset C^{\infty}_b (H)$ (each $f_n$
is also of exponential type) such that
%
%
\begin{equation}
\label{cd} f_n \to f,\qquad Df_n \to Df \qquad\mbox{in }
L^2(\mu)
\end{equation}
as $n \to\infty$. We fix $t>0$ and
write, for any $n \ge1$,
%
%
\begin{equation}
\label{frtt} f_n(X_t) - f_n(Y_t)
= \int_0^1 \bigl\langle Df_n
\bigl(r X_t + (1-r)Y_t\bigr), X_t-
Y_t \bigr\rangle\,dr.
\end{equation}
%
For a fixed $T>0$ we
will show that, as $n \to\infty$, the left-hand
side and the right-hand side of (\ref{frtt}) converge
in $L^1([\epsilon, T] \times\Omega,dt \otimes
\P)$, respectively, to the left-hand
side and the right-hand side of (\ref{tes2}) for all $\epsilon
\in(0,T]$.

We only prove convergence of the right-hand side
of (\ref{frtt}) (the convergence of the left-hand side is similar and simpler).

Fix $\epsilon\in(0,T]$.
We first consider the case in which $x$ is deterministic. We get, using
the Girsanov theorem (see Theorem~\ref{tA1}), as in the proof of
Proposition~\ref{stop},
\begin{eqnarray*}
a_n:\!&=& {\mathbb E} \biggl[ \int_{\epsilon}^T
\int_0^1 \bigl|Df_n \bigl(r
X_t + (1-r)Y_t\bigr)- Df \bigl(r X_t +
(1-r)Y_t\bigr)\bigr|\\
&&\hspace*{172pt}{}\times \bigl| (X_t- Y_t) \bigr| \,dr \,dt \biggr]
\\
&\le& M \int_{\epsilon}^T \int_0^1
{\mathbb E} \bigl|Df_n \bigl(r X_t + (1-r)Y_t
\bigr)- Df \bigl(r X_t + (1-r)Y_t\bigr) \bigr| \,dr \,dt
\\
&\le& M' \int_0^1 {\mathbb E}
\biggl[ \rho^{-1/2}_r \rho^{1/2}_r \int
_{\epsilon}^T \bigl|Df_n \bigl(r
X_t + (1-r)Y_t\bigr)\\
&&\hspace*{109pt}{} - Df \bigl(r X_t +
(1-r)Y_t\bigr)\bigr| \,dt \biggr]\,dr
\\
&\le& M \biggl(\int_0^1 {\mathbb E} \bigl[
\rho^{-1}_r \bigr] \,dr \biggr)^{1/2} \biggl(\int
_0^1 \int_{\epsilon}^T
{\mathbb E}\bigl[ \bigl|Df_n ( U_t)- Df (U_t)\bigr|^2
\bigr] \,dt \,dr \biggr)^{1/2}
\\
&\le& C \biggl( \int_{\epsilon}^T {\mathbb E}\bigl[
\bigl|Df_n ( U_t)- Df (U_t)\bigr|^2 \bigr]
\,dt \biggr)^{1/2},
\end{eqnarray*}
where $U_t$ is an OU process starting at $x$. By
\cite{DZ1}, Lemma 10.3.3, we know that, for $t>0$, the law of $U_t$ has a
positive density $\pi(t, x, \cdot)$ with respect to~$\mu$,
bounded on $[\epsilon, T] \times H$.

It easily follows [using (\ref{cd})] that
$ \int_{\epsilon}^T {\mathbb E}[ |Df_n ( U_t)-
Df (U_t)|^2 ] \,dt$ $\to0$, as $n \to\infty$,
and so $a_n \to0$.

Similarly, one proves that
\[
\int_{\epsilon}^T {\mathbb E} \biggl[ \int
_0^1 \bigl| Df \bigl(r X_t +
(1-r)Y_t\bigr)\bigr| \,dr \biggr] \,dt < \infty.
\]
Now since $\epsilon\in(0,T]$ was arbitrary, the assertion follows
for every (nonrandom) initial condition $x \in H$.

Now let us consider the case in which $x$ is an
${\mathcal F}_0$-measurable r.v. Using Remark~\ref{obvio},
analogously, we find,
with $1/p + 1/p' =1 $ and $1<p<2$,
\begin{eqnarray*}
a_n &\le& M \int_0^1 \int
_{\epsilon}^T {\mathbb E} \bigl[ \rho^{-1/p}_r
\rho^{1/p}_r \bigl|Df_n \bigl(r X_t +
(1-r)Y_t\bigr)\\
&&\hspace*{108.3pt}{} - Df \bigl(r X_t + (1-r)Y_t
\bigr)\bigr| \bigr] \,dt\,dr
\\
&\le& M' \biggl(\int_0^1 {
\mathbb E} \bigl[ \rho^{-p'/p}_r \bigr] \,dr
\biggr)^{1/p'} \biggl(\int_0^1 \int
_{\epsilon}^T {\mathbb E}\bigl[ \bigl|Df_n (
U_t)- Df (U_t)\bigr|^p \bigr] \,dt \,dr
\biggr)^{1/p}
\\
&\le& C \biggl( \int_{\epsilon}^T {\mathbb E}\bigl[
\bigl|Df_n ( U_t)- Df (U_t)\bigr|^p \bigr]
\,dt \biggr)^{1/p},
\end{eqnarray*}
where $U_t$ is an OU process such that $U_0=x$, $\P$-a.s.
Using
(\ref{rit}) with $|Df_n -
Df |^p$ instead of $f$ and $\zeta' = 2/p$, as above, we
arrive at
\[
a_n \le C_{\epsilon} \| h_0\|_{L^{{2}/({2-p})}(\mu)}^{1/p}
\biggl( \int_H \bigl|Df_n ( x)- Df
(x)\bigr|^2 \mu(dx) \biggr)^{1/2},
\]
where $h_0$ denotes the density of the law of $x$ with respect to
$\mu$.
Passing to the limit, by (\ref{frtt}) we get $a_n \to0$. Then
analogously to the case where $x$ is deterministic, we complete the
proof.
\end{pf}
\end{appendix}

\section*{Acknowledgments}

The authors thank the unknown referee for helpful comments. Priola
and R\"ockner are grateful for the hospitality during a very fruitful
stay in Pisa in June 2011 during which a substantial part of this work
was done.



\printaddresses


\begin{thebibliography}{35}

\bibitem{AlabertGyongy}
\begin{barticle}[mr]
\bauthor{\bsnm{Alabert},~\bfnm{Aureli}\binits{A.}} \AND
  \bauthor{\bsnm{Gy{\"o}ngy},~\bfnm{Istv{\'a}n}\binits{I.}}
(\byear{2001}).
\btitle{On stochastic reaction-diffusion equations with singular force term}.
\bjournal{Bernoulli}
\bvolume{7}
\bpages{145--164}.
\bid{doi={10.2307/3318606}, issn={1350-7265}, mr={1811748}}
\bptok{imsref}%
\end{barticle}
\endbibitem

\bibitem{CG}
\begin{barticle}[mr]
\bauthor{\bsnm{Chojnowska-Michalik},~\bfnm{Anna}\binits{A.}} \AND
  \bauthor{\bsnm{Goldys},~\bfnm{Beniamin}\binits{B.}}
(\byear{2001}).
\btitle{Generalized {O}rnstein--{U}hlenbeck semigroups:
  {L}ittlewood--{P}aley--{S}tein inequalities and the {P}. {A}.\ {M}eyer
  equivalence of norms}.
\bjournal{J. Funct. Anal.}
\bvolume{182}
\bpages{243--279}.
\bid{doi={10.1006/jfan.2000.3722}, issn={0022-1236}, mr={1828795}}
\bptok{imsref}%
\end{barticle}
\endbibitem

\bibitem{CG1}
\begin{barticle}[mr]
\bauthor{\bsnm{Chojnowska-Michalik},~\bfnm{Anna}\binits{A.}} \AND
  \bauthor{\bsnm{Goldys},~\bfnm{Beniamin}\binits{B.}}
(\byear{2002}).
\btitle{Symmetric {O}rnstein--{U}hlenbeck semigroups and their generators}.
\bjournal{Probab. Theory Related Fields}
\bvolume{124}
\bpages{459--486}.
\bid{doi={10.1007/s004400200222}, issn={0178-8051}, mr={1942319}}
\bptok{imsref}%
\end{barticle}
\endbibitem

\bibitem{D}
\begin{bbook}[mr]
\bauthor{\bsnm{Da~Prato},~\bfnm{Giuseppe}\binits{G.}}
(\byear{2004}).
\btitle{Kolmogorov Equations for Stochastic {PDE}s}.
\bpublisher{Birkh\"auser}, \blocation{Basel}.
\bid{doi={10.1007/978-3-0348-7909-5}, mr={2111320}}
\bptok{imsref}%
\end{bbook}
\endbibitem

\bibitem{DF}
\begin{barticle}[mr]
\bauthor{\bsnm{Da~Prato},~\bfnm{G.}\binits{G.}} \AND
  \bauthor{\bsnm{Flandoli},~\bfnm{F.}\binits{F.}}
(\byear{2010}).
\btitle{Pathwise uniqueness for a class of {SDE} in {H}ilbert spaces and
  applications}.
\bjournal{J. Funct. Anal.}
\bvolume{259}
\bpages{243--267}.
\bid{doi={10.1016/j.jfa.2009.11.019}, issn={0022-1236}, mr={2610386}}
\bptok{imsref}%
\end{barticle}
\endbibitem

\bibitem{DL}
\begin{bmisc}[auto:STB|2012/08/14|15:18:37]
\bauthor{\bsnm{Da~Prato},~\bfnm{G.}\binits{G.}} \AND
  \bauthor{\bsnm{Lunardi},~\bfnm{A.}\binits{A.}}
(\byear{2012}).
\bhowpublished{Maximal $L^2$-regularity for Dirichlet problems in Hilbert
  spaces. Preprint. Available at arXiv:\arxivurl{1201.3809}.}
\bptok{imsref}%
\end{bmisc}
\endbibitem

\bibitem{DZ}
\begin{bbook}[mr]
\bauthor{\bsnm{Da~Prato},~\bfnm{Giuseppe}\binits{G.}} \AND
  \bauthor{\bsnm{Zabczyk},~\bfnm{Jerzy}\binits{J.}}
(\byear{1992}).
\btitle{Stochastic Equations in Infinite Dimensions}.
\bseries{Encyclopedia of Mathematics and Its Applications}
\bvolume{44}.
\bpublisher{Cambridge Univ. Press}, \blocation{Cambridge}.
\bid{doi={10.1017/CBO9780511666223}, mr={1207136}}
\bptok{imsref}%
\end{bbook}
\endbibitem

\bibitem{DZ1}
\begin{bbook}[mr]
\bauthor{\bsnm{Da~Prato},~\bfnm{Giuseppe}\binits{G.}} \AND
  \bauthor{\bsnm{Zabczyk},~\bfnm{Jerzy}\binits{J.}}
(\byear{2002}).
\btitle{Second Order Partial Differential Equations in {H}ilbert Spaces}.
\bseries{London Mathematical Society Lecture Note Series}
\bvolume{293}.
\bpublisher{Cambridge Univ. Press}, \blocation{Cambridge}.
\bid{doi={10.1017/CBO9780511543210}, mr={1985790}}
\bptok{imsref}%
\end{bbook}
\endbibitem

\bibitem{Fedrizzi}
\begin{bmisc}[auto:STB|2012/08/14|15:18:37]
\bauthor{\bsnm{Fedrizzi},~\bfnm{E.}\binits{E.}}
(\byear{2009}).
\bhowpublished{Uniqueness and Flow Theorems for
solutions of SDEs with low regularity of the drift. Ph.D. thesis, Pisa}.
\bptok{imsref}%
\end{bmisc}
\endbibitem

\bibitem{FF}
\begin{barticle}[mr]
\bauthor{\bsnm{Fedrizzi},~\bfnm{E.}\binits{E.}} \AND
  \bauthor{\bsnm{Flandoli},~\bfnm{F.}\binits{F.}}
(\byear{2011}).
\btitle{Pathwise uniqueness and continuous dependence of {SDE}s with
  non-regular drift}.
\bjournal{Stochastics}
\bvolume{83}
\bpages{241--257}.
\bid{doi={10.1080/17442508.2011.553681}, issn={1744-2508}, mr={2810591}}
\bptok{imsref}%
\end{barticle}
\endbibitem

\bibitem{Ferr10}
\begin{bmisc}[auto]
\bauthor{\bsnm{Ferrario},~\bfnm{B.}\binits{B.}}
(\byear{2010}).
\bhowpublished{A note on a result of Liptser--Shiryaev.
\textit{Stoch. Anal. Appl.}
To appear.  Available at Available at \arxivurl{arXiv:1005.0237}}.
\bptok{imsref}%
\end{bmisc}
\endbibitem

\bibitem{Fe}
\begin{bmisc}[mr]
\bauthor{\bsnm{Ferrario},~\bfnm{Benedetta}\binits{B.}}
(\byear{2012}).
\bhowpublished{Uniqueness and absolute continuity for semilinear SPDE's. In
  \textit{Seminar on Stochastic Analysis, Random Fields and Applications VII}. Birkh\"auser, Boston, MA. To appear.}
\bptok{imsref}%
\end{bmisc}
\endbibitem

\bibitem{Fla}
\begin{bbook}[mr]
\bauthor{\bsnm{Flandoli},~\bfnm{Franco}\binits{F.}}
(\byear{2011}).
\btitle{Random Perturbation of {PDE}s and Fluid Dynamic Models}.
\bseries{Lecture Notes in Math.}
\bvolume{2015}.
\bpublisher{Springer}, \blocation{Heidelberg}.
\bid{doi={10.1007/978-3-642-18231-0}, mr={2796837}}
\bptok{imsref}%
\end{bbook}
\endbibitem

\bibitem{FGP}
\begin{barticle}[mr]
\bauthor{\bsnm{Flandoli},~\bfnm{F.}\binits{F.}},
  \bauthor{\bsnm{Gubinelli},~\bfnm{M.}\binits{M.}} \AND
  \bauthor{\bsnm{Priola},~\bfnm{E.}\binits{E.}}
(\byear{2010}).
\btitle{Well-posedness of the transport equation by stochastic perturbation}.
\bjournal{Invent. Math.}
\bvolume{180}
\bpages{1--53}.
\bid{doi={10.1007/s00222-009-0224-4}, issn={0020-9910}, mr={2593276}}
\bptok{imsref}%
\end{barticle}
\endbibitem

\bibitem{GG}
\begin{bincollection}[mr]
\bauthor{\bsnm{G{\c{a}}tarek},~\bfnm{D.}\binits{D.}} \AND
  \bauthor{\bsnm{Go{\l}dys},~\bfnm{B.}\binits{B.}}
(\byear{1992}).
\btitle{On solving stochastic evolution equations by the change of drift with
  application to optimal control}.
In \bbooktitle{Stochastic Partial Differential Equations and Applications
  ({T}rento, 1990)}.
\bseries{Pitman Research Notes in Mathematics Series}
\bvolume{268}
\bpages{180--190}.
\bpublisher{Longman}, \blocation{Harlow, UK}.
\bid{mr={1222696}}
\bptok{imsref}%
\end{bincollection}
\endbibitem

\bibitem{Gyongy}
\begin{barticle}[mr]
\bauthor{\bsnm{Gy{\"o}ngy},~\bfnm{Istv{\'a}n}\binits{I.}}
(\byear{1998}).
\btitle{Existence and uniqueness results for semilinear stochastic partial
  differential equations}.
\bjournal{Stochastic Process. Appl.}
\bvolume{73}
\bpages{271--299}.
\bid{doi={10.1016/S0304-4149(97)00103-8}, issn={0304-4149}, mr={1608641}}
\bptok{imsref}%
\end{barticle}
\endbibitem

\bibitem{Gyongy-Krylov}
\begin{barticle}[mr]
\bauthor{\bsnm{Gy{\"o}ngy},~\bfnm{Istv{\'a}n}\binits{I.}} \AND
  \bauthor{\bsnm{Krylov},~\bfnm{Nicolai}\binits{N.}}
(\byear{1996}).
\btitle{Existence of strong solutions for {I}t\^o's stochastic equations via
  approximations}.
\bjournal{Probab. Theory Related Fields}
\bvolume{105}
\bpages{143--158}.
\bid{doi={10.1007/BF01203833}, issn={0178-8051}, mr={1392450}}
\bptok{imsref}%
\end{barticle}
\endbibitem

\bibitem{GyongyMartinez}
\begin{barticle}[mr]
\bauthor{\bsnm{Gy{\"o}ngy},~\bfnm{Istv{\'a}n}\binits{I.}} \AND
  \bauthor{\bsnm{Mart{\'{\i}}nez},~\bfnm{Teresa}\binits{T.}}
(\byear{2001}).
\btitle{On stochastic differential equations with locally unbounded drift}.
\bjournal{Czechoslovak Math. J.}
\bvolume{51}
\bpages{763--783}.
\bid{doi={10.1023/A:1013764929351}, issn={0011-4642}, mr={1864041}}
\bptok{imsref}%
\end{barticle}
\endbibitem

\bibitem{GyongyNualart}
\begin{barticle}[mr]
\bauthor{\bsnm{Gy{\"o}ngy},~\bfnm{Istv{\'a}n}\binits{I.}} \AND
  \bauthor{\bsnm{Nualart},~\bfnm{David}\binits{D.}}
(\byear{1999}).
\btitle{On the stochastic {B}urgers' equation in the real line}.
\bjournal{Ann. Probab.}
\bvolume{27}
\bpages{782--802}.
\bid{doi={10.1214/aop/1022677386}, issn={0091-1798}, mr={1698967}}
\bptok{imsref}%
\end{barticle}
\endbibitem

\bibitem{GyongyPardoux}
\begin{barticle}[mr]
\bauthor{\bsnm{Gy{\"o}ngy},~\bfnm{Istv{\'a}n}\binits{I.}} \AND
  \bauthor{\bsnm{Pardoux},~\bfnm{{\'E}.}\binits{{\'E}.}}
(\byear{1993}).
\btitle{On the regularization effect of space-time white noise on quasi-linear
  parabolic partial differential equations}.
\bjournal{Probab. Theory Related Fields}
\bvolume{97}
\bpages{211--229}.
\bid{doi={10.1007/BF01199321}, issn={0178-8051}, mr={1240724}}
\bptok{imsref}%
\end{barticle}
\endbibitem

\bibitem{issoglio}
\begin{bmisc}[auto:STB|2012/08/14|15:18:37]
\bauthor{\bsnm{Issoglio},~\bfnm{E.}\binits{E.}}
(\byear{2011}).
\bhowpublished{Transport equations with fractal noise---existence, uniqueness
  and regularity of the solution.
\textit{Z. Anal. Anwend.} To appear.
  Available at
  arXiv:\arxivurl{1107.3788}.}
\bptok{imsref}%
\end{bmisc}
\endbibitem

\bibitem{KS}
\begin{bbook}[mr]
\bauthor{\bsnm{Karatzas},~\bfnm{Ioannis}\binits{I.}} \AND
  \bauthor{\bsnm{Shreve},~\bfnm{Steven~E.}\binits{S.~E.}}
(\byear{1991}).
\btitle{Brownian Motion and Stochastic Calculus},
\bedition{2nd} ed.
\bseries{Graduate Texts in Mathematics}
\bvolume{113}.
\bpublisher{Springer}, \blocation{New York}.
\bid{doi={10.1007/978-1-4612-0949-2}, mr={1121940}}
\bptok{imsref}%
\end{bbook}
\endbibitem

\bibitem{KR05}
\begin{barticle}[mr]
\bauthor{\bsnm{Krylov},~\bfnm{N.~V.}\binits{N.~V.}} \AND
  \bauthor{\bsnm{R{\"o}ckner},~\bfnm{M.}\binits{M.}}
(\byear{2005}).
\btitle{Strong solutions of stochastic equations with singular time dependent
  drift}.
\bjournal{Probab. Theory Related Fields}
\bvolume{131}
\bpages{154--196}.
\bid{doi={10.1007/s00440-004-0361-z}, issn={0178-8051}, mr={2117951}}
\bptok{imsref}%
\end{barticle}
\endbibitem

\bibitem{dap}
\begin{bincollection}[mr]
\bauthor{\bsnm{Kunstmann},~\bfnm{Peer~C.}\binits{P.~C.}} \AND
  \bauthor{\bsnm{Weis},~\bfnm{Lutz}\binits{L.}}
(\byear{2004}).
\btitle{Maximal {$L\sb p$}-regularity for parabolic equations, {F}ourier
  multiplier theorems and {$H\sp\infty $}-functional calculus}.
In \bbooktitle{Functional Analytic Methods for Evolution Equations}
(\beditor{\bfnm{M.}\binits{M.}~\bsnm{Iannelli}},
  \beditor{\bfnm{R.}\binits{R.}~\bsnm{Nagel}} \AND
  \beditor{\bfnm{S.}\binits{S.}~\bsnm{Piazzera}}, eds.).
\bseries{Lecture Notes in Math.}
\bvolume{1855}
\bpages{65--311}.
\bpublisher{Springer}, \blocation{Berlin}.
\bid{doi={10.1007/978-3-540-44653-8_2}, mr={2108959}}
\bptok{imsref}%
\end{bincollection}
\endbibitem

\bibitem{MR}
\begin{bbook}[mr]
\bauthor{\bsnm{Ma},~\bfnm{Zhi~Ming}\binits{Z.~M.}} \AND
  \bauthor{\bsnm{R{\"o}ckner},~\bfnm{Michael}\binits{M.}}
(\byear{1992}).
\btitle{Introduction to the Theory of (nonsymmetric) {D}irichlet Forms}.
\bpublisher{Springer}, \blocation{Berlin}.
\bid{doi={10.1007/978-3-642-77739-4}, mr={1214375}}
\bptok{imsref}%
\end{bbook}
\endbibitem

\bibitem{MV}
\begin{bincollection}[auto:STB|2012/08/14|15:18:37]
\bauthor{\bsnm{Maas},~\bfnm{J.}\binits{J.}} \AND \bauthor{\bparticle{van}
  \bsnm{Neerven},~\bfnm{J.~M. A.~M.}\binits{J.~M. A.~M.}}
(\byear{2011}).
\btitle{Gradient estimates and domain identification for analytic
  Ornstein--Uhlenbeck operators}.
In \bbooktitle{Progress in Nonlinear Differential Equations and Their Applications}
\bvolume{60}
\bpages{463--477}.
\bpublisher{Springer}, \blocation{Basel}.
\bptok{imsref}%
\end{bincollection}
\endbibitem

\bibitem{MPRS}
\begin{barticle}[mr]
\bauthor{\bsnm{Metafune},~\bfnm{Giorgio}\binits{G.}},
  \bauthor{\bsnm{Pr{\"u}ss},~\bfnm{Jan}\binits{J.}},
  \bauthor{\bsnm{Rhandi},~\bfnm{Abdelaziz}\binits{A.}} \AND
  \bauthor{\bsnm{Schnaubelt},~\bfnm{Roland}\binits{R.}}
(\byear{2002}).
\btitle{The domain of the {O}rnstein--{U}hlenbeck operator on an {$L\sp
  p$}-space with invariant measure}.
\bjournal{Ann. Sc. Norm. Super. Pisa Cl. Sci. (5)}
\bvolume{1}
\bpages{471--485}.
\bid{issn={0391-173X}, mr={1991148}}
\bptok{imsref}%
\end{barticle}
\endbibitem

\bibitem{PrevRoeckner}
\begin{bbook}[mr]
\bauthor{\bsnm{Pr{\'e}v{\^o}t},~\bfnm{Claudia}\binits{C.}} \AND
  \bauthor{\bsnm{R{\"o}ckner},~\bfnm{Michael}\binits{M.}}
(\byear{2007}).
\btitle{A Concise Course on Stochastic Partial Differential Equations}.
\bseries{Lecture Notes in Math.}
\bvolume{1905}.
\bpublisher{Springer}, \blocation{Berlin}.
\bid{mr={2329435}}
\bptok{imsref}%
\end{bbook}
\endbibitem

\bibitem{P}
\begin{barticle}[auto:STB|2012/08/14|15:18:37]
\bauthor{\bsnm{Priola},~\bfnm{E.}\binits{E.}}
(\byear{2012}).
\btitle{Pathwise uniqueness for singular SDEs driven by stable
  processes}.
\bjournal{Osaka J. Math.}
\bvolume{49}
\bpages{421--447}.
\bptok{imsref}%
\end{barticle}
\endbibitem

\bibitem{RY}
\begin{bbook}[mr]
\bauthor{\bsnm{Revuz},~\bfnm{Daniel}\binits{D.}} \AND
  \bauthor{\bsnm{Yor},~\bfnm{Marc}\binits{M.}}
(\byear{1999}).
\btitle{Continuous Martingales and {B}rownian Motion},
\bedition{3rd} ed.
\bseries{Grundlehren der Mathematischen Wissenschaften [Fundamental Principles
  of Mathematical Sciences]}
\bvolume{293}.
\bpublisher{Springer}, \blocation{Berlin}.
\bid{mr={1725357}}
\bptok{imsref}%
\end{bbook}
\endbibitem

\bibitem{S}
\begin{barticle}[mr]
\bauthor{\bsnm{Schmuland},~\bfnm{B.}\binits{B.}}
(\byear{1992}).
\btitle{Dirichlet forms with polynomial domain}.
\bjournal{Math. Japon.}
\bvolume{37}
\bpages{1015--1024}.
\bid{issn={0025-5513}, mr={1196376}}
\bptok{imsref}%
\end{barticle}
\endbibitem

\bibitem{tre}
\begin{barticle}[mr]
\bauthor{\bsnm{Tanaka},~\bfnm{Hiroshi}\binits{H.}},
  \bauthor{\bsnm{Tsuchiya},~\bfnm{Masaaki}\binits{M.}} \AND
  \bauthor{\bsnm{Watanabe},~\bfnm{Shinzo}\binits{S.}}
(\byear{1974}).
\btitle{Perturbation of drift-type for {L}\'evy processes}.
\bjournal{J. Math. Kyoto Univ.}
\bvolume{14}
\bpages{73--92}.
\bid{issn={0023-608X}, mr={0368146}}
\bptok{imsref}%
\end{barticle}
\endbibitem

\bibitem{Ver}
\begin{barticle}[mr]
\bauthor{\bsnm{Veretennikov},~\bfnm{A.~Ju.}\binits{A.~J.}}
(\byear{1980}).
\btitle{Strong solutions and explicit formulas for solutions of stochastic
  integral equations}.
\bjournal{Mat. Sb. (N.S.)}
\bvolume{111}
\bpages{434--452, 480}.
\bid{issn={0368-8666}, mr={0568986}}
\bptok{imsref}%
\end{barticle}
\endbibitem

\bibitem{Za}
\begin{barticle}[mr]
\bauthor{\bsnm{Zhang},~\bfnm{Xicheng}\binits{X.}}
(\byear{2005}).
\btitle{Strong solutions of {SDES} with singular drift and {S}obolev diffusion
  coefficients}.
\bjournal{Stochastic Process. Appl.}
\bvolume{115}
\bpages{1805--1818}.
\bid{doi={10.1016/j.spa.2005.06.003}, issn={0304-4149}, mr={2172887}}
\bptok{imsref}%
\end{barticle}
\endbibitem

\bibitem{Zv74}
\begin{barticle}[mr]
\bauthor{\bsnm{Zvonkin},~\bfnm{A.~K.}\binits{A.~K.}}
(\byear{1974}).
\btitle{A transformation of the phase space of a diffusion process that will
  remove the drift}.
\bjournal{Mat. Sb. (N.S.)}
\bvolume{93}
\bpages{129--149, 152}.
\bid{mr={0336813}}
\bptok{imsref}%
\end{barticle}
\endbibitem

\end{thebibliography}
\end{document}